\documentclass[12pt, reqno]{amsart}

\usepackage[utf8]{inputenc}
\usepackage[T1]{fontenc}

\usepackage{  amssymb,  mathtools,  cite,  enumerate,  color,  hyperref,  amsfonts,  amsmath,  amsthm,  setspace,  tikz,  verbatim,  booktabs,  multirow,  enumitem}

\usepackage{charter}

\usepackage[a4paper,margin=1.9cm,top=2.5cm,bottom=2.5cm,centering,vcentering]{geometry}
\numberwithin{equation}{section}

\definecolor{ao(english)}{rgb}{0.0, 0.5, 0.0}

\hypersetup{colorlinks=true, linkcolor=ao(english),citecolor=ao(english)}

\usepackage[normalem]{ulem}

\theoremstyle{plain}
\newtheorem{theorem}{Theorem}[section]
\newtheorem{corollary}[theorem]{Corollary}
\newtheorem*{corollary*}{Corollary}

\newtheorem{lemma}[theorem]{Lemma}
\newtheorem{proposition}[theorem]{Proposition}

\theoremstyle{definition}
\newtheorem{definition}[theorem]{Definition}
\newtheorem*{definition*}{Definition}
\newtheorem*{def*}{Definition}

\newtheorem*{example*}{Example}

\newtheorem*{theorem*}{Theorem}

\theoremstyle{remark}
\newtheorem*{remark}{Remark}

\newcommand{\C}{\mathbb{C}}

\newcommand\thm[1]{\ref{thm:#1}}

\newcommand\corol[1]{\ref{cor:#1}}

\newcommand\propo[1]{\ref{propo:#1}}
\newcommand\eqn[1]{(\ref{eq:#1})}
\newcommand\sect[1]{\ref{sec:#1}}

\newcommand\eqnref[1]{(\ref{eq:#1})}
\newcommand\subsect[1]{\ref{subsec:#1}}

\newcommand{\beq}{\begin{equation}}
\newcommand{\eeq}{\end{equation}}

\allowdisplaybreaks

\begin{document}

\title[2- and 3-Dissections of Mock Theta Functions]{2- and 3-Dissections of Second-, Sixth-, and Eighth-Order Mock Theta Functions}

\author[F. Garvan]{Frank Garvan}
\address[F. Garvan]{Department of Mathematics, University of Florida, P.O. Box 118105, Gainesville, FL 32611-8105, USA}
\email{fgarvan@ufl.edu}

\author[H. Nath]{Hemjyoti Nath}
\address[H. Nath]{Department of Mathematics, University of Florida, P.O. Box 118105, Gainesville, FL 32611-8105, USA}
\email{h.nath@ufl.edu}




\begin{abstract}
In this paper, we develop a systematic method for obtaining and proving $m$-dissections of mock theta functions. In 2014, Hickerson and Mortenson showed how to derive and prove identities for Ramanujan's mock theta functions and Hecke-type indefinite theta series using Appell--Lerch sums. We build on their transformation formula method, combining it with symbolic computations and algorithms for the theory of modular functions. We focus exclusively on the cases of 2- and 3-dissections.
\end{abstract}

\maketitle

\tableofcontents

\bigskip   

\section{Introduction}

Ramanujan, in his last letter to Hardy, introduced a collection of seventeen functions that are now known as the \emph{mock theta functions}. Each of these functions is represented by a \(q\)-series convergent for \( |q|<1 \). Although these series are not theta functions in the classical sense, they exhibit asymptotic behavior analogous to that of ordinary theta functions. Ramanujan further classified these functions according to their ``order,'' though he did not provide a precise definition of this terminology.

Ramanujan's original list consists of four third-order mock theta functions together with several associated identities, ten fifth-order mock theta functions and identities, and three seventh-order mock theta functions, which he remarked were unrelated. A substantial number of additional mock theta identities were later discovered in Ramanujan's lost notebook \cite{Ramanujan1988}. Among these were the celebrated identities for the fifth-order mock theta functions, now commonly referred to as the mock theta conjectures, which were subsequently proved by Hickerson \cite{hickerson1988proof}.

The lost notebook also contained the tenth-order mock theta functions and their associated identities \cite{Choi1999, choi2000tenth, choi2007tenth}, as well as the sixth-order mock theta functions and related identities \cite{andrews1991ramanujan, berndt2007sixth}. Although Hickerson established identities for the seventh-order mock theta functions analogous to the mock theta conjectures \cite{hickerson1988seventh}, the seventh-order functions themselves were notably absent from the lost notebook. It has been conjectured that portions of the manuscript containing these functions may have been lost \cite[p.~287]{andrews2018ramanujan}.

In this paper, we investigate dissections of mock theta functions, a topic that appears to have received relatively little systematic attention in the literature. Numerous identities for mock theta functions of various orders have been established by different authors, some of which implicitly yield \(2\)- and \(3\)-dissections. However, these identities remain scattered throughout the literature. Our investigation begins with the identities presented in \cite{mortenson2024ramanujan}, which serve as a starting point for the results established in this paper.

For the convenience of the reader, we now recall the definition of an \(m\)-dissection~\cite[Section 1.9]{hir}.

Given a series
\[
A(q)=\sum_{n} a(n)q^n,
\]
and an integer \(m>1\), the \emph{\(m\)-dissection} of \(A(q)\) is the decomposition
\[
A(q)=A_0(q)+A_1(q)+\cdots+A_{m-1}(q),
\]
where, for each integer \(r\) with \(0\le r\le m-1\),
\[
A_r(q)
=
q^r \sum_{n} a(mn+r)q^{mn}.
\]

We begin by recalling a few identities recently established by Mortenson.

Before stating these identities, we introduce some standard notation. Throughout, we set
\[
q := e^{2\pi i \tau}, 
\qquad 
\text{where } 
\tau \in \mathbb{H} := \{ z \in \mathbb{C} \mid \operatorname{Im}(z) > 0 \}.
\]

We recall the $q$-Pochhammer symbols. For a nonnegative integer $n$,
\[
(x;q)_n := \prod_{i=0}^{n-1} (1-q^i x),
\qquad
(x;q)_\infty := \prod_{i\ge0} (1-q^i x),
\qquad |q|<1.
\]

Furthermore, we define
\[
\Theta(x;q)
:= (x)_\infty (q/x)_\infty (q)_\infty
= \sum_{n\in\mathbb{Z}} (-1)^n q^{\binom{n}{2}} x^n,
\]
where the equality between the infinite product and series representations follows from Jacobi's triple product identity
\cite[Theorem~1.3.3]{berndt2006number}.

Let $a$ and $m$ be integers with $m>0$. We adopt the notation of Mortenson \cite{mortenson2024ramanujan}.
\begin{equation}\label{theta1}
\Theta_{a,m} := \Theta(q^a;q^m),
\qquad
\Theta_m := \prod_{i\ge1} (1-q^{mi}),
\qquad
\overline{\Theta}_{a,m} := \Theta(-q^a;q^m).
\end{equation}

We now recall the following identities due to Mortenson (see Theorem~2.1 of \cite{mortenson2024ramanujan}).

\begin{align*}
q B_2(q) - 2A_2(-q^4)
&= q\,\dfrac{\Theta_2 \Theta_4^{5} \Theta_{16}^2}
           {\Theta_1^2 \Theta_8^{5}}, \\
q B_2(q) + \dfrac{1}{2}\mu_2(q^4)
&= \dfrac{1}{2}\,
   \dfrac{\Theta_2 \Theta_4^{3} \Theta_8}
        {\Theta_1^{2} \Theta_{16}^{2}}, \\
q B_2(q) + \dfrac{1}{4}\mu_2(q^4) - A_2(-q^4)
&= \dfrac{1}{4}\,
   \dfrac{\Theta_2^{6} \Theta_4^{3}}
        {\Theta_1^{4} \Theta_8^{4}}.
\end{align*}

Here, $A_2(q)$, $B_2(q)$, and $\mu_2(q)$ denote the second-order mock theta functions defined in \eqref{A2}--\eqref{mu2}.                          

At first glance, these identities do not appear to yield the
$2$-dissections of $B_2(q)$ directly. However, by applying the
following $2$-dissection formulas:
\begin{align}
\frac{1}{\Theta_1^2}
&=
\frac{\Theta_8^5}
{\Theta_2^5\,\Theta_{16}^2}
+2q\,\frac{\Theta_4^2\,\Theta_{16}^2}
{\Theta_2^5\,\Theta_8},
\label{eq:theta-inv1-sq}
\\[2mm]
\frac{1}{\Theta_1^4}
&=
\frac{\Theta_4^{14}}
{\Theta_2^{14}\,\Theta_8^4}
+4q\,\frac{\Theta_4^2\,\Theta_8^4}
{\Theta_2^{10}},
\label{eq:theta-inv1-fourth}
\end{align}
these identities can be proved using the $\theta$-Step Algorithm described in
\sect{meth}. We rewrite them as identities involving generalized eta-products on
$\Gamma_1(16)$ and $\Gamma_1(8)$, respectively, and find that
$B=-4$ and $B=-2$, respectively. We verify that the first five and three terms
of the $q$-expansions on both sides agree and additionally check the identities
up to $O(q^{36})$ and $O(q^{18})$, respectively. This proves
\eqn{theta-inv1-sq} and \eqn{theta-inv1-fourth}.

Substituting these identities into the formulas above yields the desired
$2$-dissections for $B_2(q)$, namely,

\begin{align*}
    B_2(q) & = \dfrac{2}{q}\, A_2(-q^4) + \dfrac{\Theta_4^5}{\Theta_2^4} + 2q\, \dfrac{\Theta_4^7\Theta_{16}^4}{\Theta_2^4\Theta_8^6},\\
    B_2(q) & = -\dfrac{1}{2q}\, \mu_2(q^4) + \dfrac{1}{2q}\, \dfrac{\Theta_4^3\Theta_8^6}{\Theta_2^4 \Theta_{16}^4} + \dfrac{\Theta_4^5}{\Theta_2^4},\\
    B_2(q) & = - \dfrac{1}{4q}\, \mu_2(q^4) + \dfrac{1}{q}\,A_2(-q^4) + \dfrac{1}{4q}\, \dfrac{\Theta_{4}^{17}}{\Theta_2^8\Theta_8^8} + \dfrac{\Theta_4^5}{\Theta_2^4}.
\end{align*}

There are many other identities in their paper (see Theorems~2.2, 2.3, and 2.6 of \cite{mortenson2024ramanujan}) that do not explicitly appear as $2$- or $3$-dissections. In this paper, we develop a systematic method for deriving such dissections.

Recently, the second author and Das~\cite{nath2025infinite} investigated infinite families 
of congruences for the second-order mock theta function $B_2(q)$. In their work, 
they called for analytic proofs of three identities involving the 
second-order mock theta functions $A_2(q)$, $B_2(q)$, and $\mu_2(q)$. Two of these identities are
\begin{align*}
    \sum_{n \geq 0} P_{A_2}(3n + 1) q^n 
    &= \dfrac{\Theta_{2}^4 \Theta_{3}^2 \Theta_{4}}
            {\Theta_{1}^5 \Theta_{6}},\\
    \sum_{n \geq 0} P_{\mu_2}(3n + 1) q^n 
    &= -\,\dfrac{\Theta_{2}^7 \Theta_{12}^2}
            {\Theta_{1} \Theta_{4}^6 \Theta_{6}}.
\end{align*}

They further verified, using symbolic computations in \texttt{Maple}, 
the following identity for $B_2(q)$, although a complete proof was not 
provided:
\begin{equation*}
    \sum_{n \geq 0} P_{B_2}(3n) q^n 
    = \dfrac{\Theta_{2}^7 \Theta_{3}^2}
            {\Theta_{1}^6 \Theta_{4} \Theta_{6}}.
\end{equation*}

Here $P_f(n)$ denotes the coefficient of $q^n$ in the $q$-series expansion of the corresponding mock theta function $f(q)$. We adopt this notation throughout the paper. In this paper, we derive the complete 3-dissections of all second-order mock theta functions, namely $A_2(q)$, $B_2(q)$, and $\mu_2(q)$.

Very recently, Cai et al.~\cite{cai2026some} also established these 
identities using the $(p,k)$-parametrization of theta functions. Our 
approach differs substantially: we combine symbolic computations in 
\texttt{Maple} with $q$-series techniques developed by Hickerson and 
Mortenson, together with methods from the theory of modular forms.

We believe that this framework provides a promising and largely underutilized method for exploring mock theta functions and related future research directions. The paper is organized as follows. In Section~\ref{sec:catres}, we give the definitions of the mock theta functions required throughout the paper (see Section~\ref{subsec:HMcat}) and state our main results (see Section~\ref{m-Dissections}). In Section~\ref{sec:proptm}, we discuss some properties of theta functions and Appell--Lerch sums. In Section~\ref{sec:meth}, we review a method used to prove theta function identities through an example, followed by a general strategy for finding and proving the $m$-dissection of a mock theta function, together with an example. In Section~\ref{sec:2disspf}, we provide the proofs of Theorems~\ref{thm:ord2diss2}, \ref{thm:ord6diss2}, and \ref{thm:ord8diss2}, which give complete 2-dissections of many mock theta functions of orders 2, 6, and 8. In Section~\ref{sec:3disspf}, we give the proofs of Theorems~\ref{thm:ord2diss3}, \ref{thm:ord6diss3}, and \ref{thm:ord8diss3}, which give the corresponding 3-dissections. Finally, in Section~\ref{Concluding remarks}, we present concluding remarks and further observations.

\section{Catalog of Mock Theta Functions and Dissection Results}
\label{sec:catres}           

In this section, we recall the definitions and Appell--Lerch sum identities for the second, sixth, and eighth order mock theta functions from Hickerson and Mortenson \cite[Section~5]{hickerson2014hecke}, and state our main $m$-dissection results. The Appell--Lerch sum $m(x,q,z)$ is defined below in Definition~\ref{def:1p1}.

\subsection{Hickerson and Mortenson's Mock Theta Function Catalog}
\label{subsec:HMcat}

\subsubsection*{The Second-Order Mock Theta Functions}
\begin{align}
    A_2(q) 
    &= \sum_{n \geq 0}\frac{q^{n+1}(-q^2;q^2)_n}{(q;q^2)_{n+1}}
       =: \sum_{n \geq 0} P_{A_2}(n) q^n 
       = -m(q,q^4,q^2), \label{A2}\\
    B_2(q) 
    &= \sum_{n \geq 0} \frac{q^n(-q;q^2)_n}{(q;q^2)_{n+1}}
       =: \sum_{n \geq 0} P_{B_2}(n) q^n 
       = -\frac{m(1,q^4,q^3)}{q},\label{B2}\\
    \mu_2(q) 
    &= \sum_{n \geq 0} 
       \frac{(-1)^n q^{n^2}(q;q^2)_n}{(-q^2;q^2)_n^2}
       =: \sum_{n \geq 0} P_{\mu_2}(n) q^n 
       = 4m(-q,q^4,-1) 
         - \frac{\Theta_2^8}{\Theta_1^3\Theta_4^4}.\label{mu2}
\end{align}

\subsubsection*{The Sixth-Order Mock Theta Functions}

\begin{align}
    \phi_6(q) 
    &= \sum_{n \geq 0}
       \frac{(-1)^n q^{n^2}(q;q^2)_n}{(-q;q)_{2n}}
       =: \sum_{n \geq 0} P_{\phi_6}(n) q^n = 2m(q,q^3,-1), \label{phi6}\\
    \psi_6(q) 
    &= \sum_{n \geq 0} 
       \frac{(-1)^n q^{(n+1)^2}(q;q^2)_n}{(-q;q)_{2n+1}}
       =: \sum_{n \geq 0} P_{\psi_6}(n) q^n = m(1,q^3,-q), \label{psi6}\\
    \rho_6(q) 
    &= \sum_{n \geq 0} 
       \frac{q^{n(n+1)/2}(-q;q)_n}{(q;q^2)_{n+1}}
       =: \sum_{n \geq 0} P_{\rho_6}(n) q^n = -\dfrac{1}{q}m(1,q^6,q), \label{rho6}\\
    \sigma_6(q) 
    &= \sum_{n \geq 0} 
       \frac{q^{(n+1)(n+2)/2}(-q;q)_n}{(q;q^2)_{n+1}}
       =: \sum_{n \geq 0} P_{\sigma_6}(n) q^n = -m(q^2,q^6,q), \label{sigma6}\\
    \lambda_6(q) 
    &= \sum_{n \geq 0} 
       \frac{(-1)^n q^n(q;q^2)_n}{(-q;q)_n}
       =: \sum_{n \geq 0} P_{\lambda_6}(n) q^n = \dfrac{2}{q}m(1,q^6,-q^2) + \dfrac{\Theta_1\Theta_3\Theta_{12}}{\Theta_4\Theta_6}, \label{lambda6}\\
    \mu_6(q)
    &= \dfrac{1}{2} + \dfrac{1}{2}\sum_{n \geq 0}\dfrac{(-1)^nq^{n+1}(1+q^n)(q;q^2)_n}{(-q;q)_{n+1}}
    =: \sum_{n \geq 0} P_{\mu_6}(n) q^n 
    = 2m(q^2,q^6,-1) - \dfrac{\Theta_1^2\Theta_3^2}{2\Theta_2^2\Theta_6}, \label{mu6}\\
    \psi_{-_6}(q)
    &=\sum_{n \geq 1} \dfrac{q^n(-q;q)_{2n-2}}{(q;q^2)}
    =: \sum_{n \geq 0} P_{\psi_{-_6}}(n) q^n
    =-\dfrac{1}{2}m(1,q^3,q) + \dfrac{q}{2}\dfrac{\Theta_6^3}{\Theta_1\Theta_2}. \label{psiminus6}
\end{align}

\subsubsection*{The Eighth-Order Mock Theta Functions}

\begin{align}
    S_{0,8}(q) 
&= \sum_{n \ge 0} \frac{q^{n^2}(-q;q^2)_n}{(-q^2;q^2)_n} =: \sum_{n \geq 0} P_{S_{0,8}}(n) q^n = 2m(-q^3,q^8,-1) 
  + q\dfrac{\overline{\Theta}_{1,8}\Theta_{2,8}^2}{\Theta_{3,8}^2}, \label{S08}\\
  S_{1,8}(q)
&= \sum_{n \ge 0} \frac{q^{n(n+2)}(-q;q^2)_n}{(-q^2;q^2)_n}=: \sum_{n \geq 0} P_{S_{1,8}}(n) q^n = -\dfrac{2}{q}m(-q,q^8,-1)
  + \dfrac{1}{q}\dfrac{\overline{\Theta}_{3,8}\Theta_{2,8}^2}{\Theta_{1,8}^2}, \label{S18}\\
  U_{0,8}(q)
&= \sum_{n \ge 0} \frac{q^{n^2}(-q;q^2)_n}{(-q^4;q^4)_n} =: \sum_{n \geq 0} P_{U_{0,8}}(n) q^n
= 2m(-q,q^4,-1),\label{U08}\\
U_{1,8}(q)
&= \sum_{n \ge 0} \frac{q^{(n+1)^2}(-q;q^2)_n}{(-q^2;q^4)_{n+1}} =: \sum_{n \geq 0} P_{U_{1,8}}(n) q^n
= -m(-q,q^4,-q^2),\label{U18}\\
    V_{0,8}(q) 
    &= -1 + 2 \sum_{n \ge 0} 
       q^{n^2} \frac{(-q; q^2)_n}{(q; q^2)_n}
       =: \sum_{n \geq 0} P_{V_{0,8}}(n) q^n = -\dfrac{2}{q}m(1,q^8,q) - \dfrac{\Theta_2^3\Theta_4}{\Theta_1^2\Theta_8},\label{V08}\\
    V_{1,8}(q) 
    &= \sum_{n \ge 0} 
       \frac{q^{(n+1)^2}(-q;q^2)_n}{(q;q^2)_{n+1}}
       =: \sum_{n \geq 0} P_{V_{1,8}}(n) q^n = -m(q^2,q^8,q).\label{V18}
\end{align}

\subsection{The $m$-Dissection Results}\label{m-Dissections}
We state the complete $m$-dissections  for the second, 
sixth and eigth order mock theta functions for both $m=2$ and $3$.

\subsubsection*{The $2$-Dissections}

\begin{theorem}
\label{thm:ord2diss2}
The following representations hold:
\begin{align}
A_2(q)
&=
\frac{q}{2}\,S_{1,8}(q^2)
-\frac{1}{2}\,S_{0,8}(q^2)
+\frac{\Theta_4^9\,\overline{\Theta}_{6,16}}
{2\,\Theta_2^5\,\Theta_8^4}
+\frac{q}{2}\,
\frac{\Theta_4^9\,\overline{\Theta}_{2,16}}
{\Theta_2^5\,\Theta_8^4},
\label{eq:A2diss2}
\\[2mm]
B_2(q)
&=
-\frac{U_{0,8}(q^4)}{q}
+\frac{1}{q}\,
\frac{\Theta_8^6\,\Theta_{16}}
{\Theta_2^2\,\Theta_4^2\,\Theta_{32}^2}
+\frac{\Theta_4^5}{\Theta_2^4},
\label{eq:B2diss2}
\\[2mm]
\mu_2(q)
&=
2\,S_{0,8}(q^2)
+2q\,S_{1,8}(q^2)
-\dfrac{\Theta_2\Theta_8^5\overline{\Theta}_{6,16}}{\Theta_4^4\Theta_{16}^2} - q\,\dfrac{\Theta_2\Theta_8^5\overline{\Theta}_{2,16}}{\Theta_4^4\Theta_{16}^2} \nonumber\\
&-2q\,\dfrac{\Theta_2\Theta_{16}^2\overline{\Theta}_{6,16}}{\Theta_4^2\Theta_8}-2q^2\,\dfrac{\Theta_2\Theta_{16}^2\overline{\Theta}_{2,16}}{\Theta_4^4\Theta_8}.
\label{eq:mu2diss2}
\end{align}
\end{theorem}

\begin{theorem}
\label{thm:ord6diss2}
The following representations hold:
\begin{align}
\rho_6(q)
&=
-\frac{U_{0,8}(q^6)}{q}
+\frac{1}{q}\,
\frac{\Theta_4\,\Theta_8^2\,\Theta_{12}^3}
{\Theta_2^2\,\Theta_6\,\Theta_{16}\,\Theta_{48}}
+\frac{\Theta_4^3\,\Theta_6^2}
{\Theta_2^3\,\Theta_{12}},
\label{eq:rho6diss2}
\\[2mm]
\lambda_6(q)
&=
\frac{\mu_2(q^6)}{q}
+\frac{\Theta_4^3\,\Theta_6^2}
{\Theta_2^3\,\Theta_{12}}
-\frac{\Theta_8^4\,\Theta_{12}^5}
{q\,\Theta_2\,\Theta_4^3\,\Theta_{24}^4},
\label{eq:lambda6diss2}
\\[2mm]
\sigma_6(q)
&=
-\frac{\phi_6(q^2)}{2}
+\frac{q^2\,\Theta_4^8\,\Theta_6\,\Theta_{24}^4}
{2\,\Theta_2^4\,\Theta_8^4\,\Theta_{12}^4}
+\frac{q\,\Theta_4^2\,\Theta_{12}^2}
{\Theta_2^2\,\Theta_6}
+\frac{\Theta_8^4\,\Theta_{12}^8}
{2\,\Theta_4^4\,\Theta_6^3\,\Theta_{24}^4},
\label{eq:sigma6diss2}
\\[2mm]
\mu_6(q)
&=
\phi_6(q^2)
-\frac{\Theta_8^4\,\Theta_{12}^8}
{2\,\Theta_4^4\,\Theta_6^3\,\Theta_{24}^4}
+q\,\frac{\Theta_4^2\,\Theta_{12}^2}
{\Theta_2^2\,\Theta_6}
-\frac{q^2}{2}\,
\frac{\Theta_4^8\,\Theta_6\,\Theta_{24}^4}
{\Theta_2^4\,\Theta_8^4\,\Theta_{12}^4}.
\label{eq:mu6diss2}
\end{align}
\end{theorem}

\begin{theorem}
\label{thm:ord8diss2}
The following representations hold:
\begin{align}
V_{0,8}(q)
&=
2q\,B_2(q^2)
+\frac{\Theta_4^8}
{\Theta_2^4\,\Theta_8^3},
\label{eq:V08diss2}
\\[2mm]
V_{1,8}(q)
&=
A_2(q^2)
+q\,\frac{\Theta_8^3}
{\Theta_2\,\Theta_4}.
\label{eq:V18diss2}
\end{align}
\end{theorem}

The following corollaries are immediate consequences of the preceding theorems.

\begin{corollary}
\label{cor:A2diss2}
The following identities hold:
\begin{align}
\sum_{n \ge 0} P_{A_2}(2n)\,q^n
&=
-\frac{S_{0,8}(q)}{2}
+\frac{\Theta_2^9\,\overline{\Theta}_{3,8}}
{2\,\Theta_1^5\,\Theta_4^4},
\label{eq:A2n}
\\[2mm]
\sum_{n \ge 0} P_{A_2}(2n+1)\,q^n
&=
\frac{S_{1,8}(q)}{2}
+\frac{\Theta_2^9\,\overline{\Theta}_{1,8}}
{2\,\Theta_1^5\,\Theta_4^4}.
\label{eq:A2nplus1}
\end{align}
\end{corollary}

\begin{corollary}
\label{cor:B2ord}
The following identities hold:
\begin{align}
\sum_{n \ge 0} P_{B_2}(2n)\,q^n
&=
\frac{\Theta_2^5}
{\Theta_1^4},
\label{eq:B2n}
\\[2mm]
\sum_{n \ge 0} P_{B_2}(2n+1)\,q^n
&=
-\frac{U_{0,8}(q^2)}{q}
+\frac{1}{q}\,
\frac{\Theta_4^6\,\Theta_8}
{\Theta_1^2\,\Theta_2^2\,\Theta_{16}^2}.
\label{eq:B2nplus1}
\end{align}
\end{corollary}

\begin{corollary}
\label{cor:mu2ord}
The following identities hold:
\begin{align}
\sum_{n \ge 0} P_{\mu_2}(2n)\,q^n
&=
2\,S_{0,8}(q)-\dfrac{\Theta_1\Theta_4^5\overline{\Theta}_{3,8}}{\Theta_2^4\Theta_8^2} - 2q\,\dfrac{\Theta_1\Theta_8^2\overline{\Theta}_{1,8}}{\Theta_2^4\Theta_4},
\label{eq:mu2n}
\\[2mm]
\sum_{n \ge 0} P_{\mu_2}(2n+1)\,q^n
&=
2\,S_{1,8}(q) - \dfrac{\Theta_1\Theta_4^5\overline{\Theta}_{1,8}}{\Theta_2^4\Theta_8^2} - 2\,\dfrac{\Theta_1\Theta_8^2\overline{\Theta}_{3,8}}{\Theta_2^2\Theta_4}.
\label{eq:mu2nplus1}
\end{align}
\end{corollary}

\begin{corollary}
\label{cor:rho6diss2}
The following identities hold:
\begin{align}
\sum_{n \ge 0} P_{\rho_6}(2n)\,q^n
&=
\frac{\Theta_2^3\,\Theta_3^2}
{\Theta_1^3\,\Theta_6},
\label{eq:rho2n}
\\[2mm]
\sum_{n \ge 0} P_{\rho_6}(2n+1)\,q^n
&=
-\frac{U_{0,8}(q^3)}{q}
+\frac{1}{q}\,
\frac{\Theta_2\,\Theta_4^2\,\Theta_6^3}
{\Theta_1^2\,\Theta_3\,\Theta_8\,\Theta_{24}}.
\label{eq:rho2nplus1}
\end{align}
\end{corollary}

\begin{corollary}
\label{cor:sigma6diss2}
The following identities hold:
\begin{align}
\sum_{n \ge 0} P_{\sigma_6}(2n)\,q^n
&=
-\frac{\phi_6(q)}{2}
+\frac{q\,\Theta_2^8\,\Theta_3\,\Theta_{12}^4}
{2\,\Theta_1^4\,\Theta_4^4\,\Theta_6^4}
+\frac{\Theta_4^4\,\Theta_6^8}
{2\,\Theta_2^4\,\Theta_3^3\,\Theta_{12}^4},
\label{eq:sigma2n}
\\[2mm]
\sum_{n \ge 0} P_{\sigma_6}(2n+1)\,q^n
&=
\frac{\Theta_2^2\,\Theta_6^2}
{\Theta_1^2\,\Theta_3}.
\label{eq:sigma2nplus1}
\end{align}
\end{corollary}

\begin{corollary}
\label{cor:lambda6diss2}
The following identities hold:
\begin{align}
\sum_{n \ge 0} P_{\lambda_6}(2n)\,q^n
&=
\frac{\Theta_2^3\,\Theta_3^2}
{\Theta_1^3\,\Theta_6},
\label{eq:lambda2n}
\\[2mm]
\sum_{n \ge 0} P_{\lambda_6}(2n+1)\,q^n
&=
\frac{\mu_2(q^3)}{q}
+\frac{\Theta_4^4\,\Theta_6^5}
{q\,\Theta_1\,\Theta_2^3\,\Theta_{12}^4}.
\label{eq:lambda2nplus1}
\end{align}
\end{corollary}

\begin{corollary}
\label{cor:mu6diss2}
The following identities hold:
\begin{align}
\sum_{n \ge 0} P_{\mu_6}(2n)\,q^n
&=
\phi_6(q)
+\frac{\Theta_1\,\Theta_3\,\Theta_4^5\,\Theta_{12}^5}
{\Theta_2^2\,\Theta_6^2\,\Theta_8^2\,\Theta_{24}^2}
+4q\,\frac{\Theta_1\,\Theta_3\,\Theta_8^2\,\Theta_{24}^2}
{\Theta_4\,\Theta_{12}},
\label{eq:mu62n}
\\[2mm]
\sum_{n \ge 0} P_{\mu_6}(2n+1)\,q^n
&=
\frac{\Theta_2^2\,\Theta_6^2}
{\Theta_1^2\,\Theta_3}.
\label{eq:mu62nplus1}
\end{align}
\end{corollary}

\begin{corollary}
\label{cor:V08diss2}
The following identities hold:
\begin{align}
\sum_{n \ge 0} P_{V_{0,8}}(2n)\,q^n
&=
\frac{\Theta_2^8}
{\Theta_1^4\,\Theta_4^3},
\label{eq:V02n}
\\[2mm]
\sum_{n \ge 0} P_{V_{0,8}}(2n+1)\,q^n
&=
\frac{4}{q}\,A_2(-q^4)
+2\,\frac{\Theta_2\,\Theta_4^5\,\Theta_{16}^2}
{\Theta_1^2\,\Theta_8^5}.
\end{align}
\end{corollary}


\begin{corollary}
\label{cor:V18diss2}      
For all $n \ge 0$, we have
    \begin{align}
    \sum_{n \ge 0} P_{V_{1,8}}(2n) q^{n} & = \dfrac{q}{2}S_{1,8}(q^2) - \dfrac{S_{0,8}(q^2)}{2} + \dfrac{\Theta_4^9}{2\Theta_1\Theta_2^3\Theta_{8}^4},\label{eq:V_18 2n}\\
    \sum_{n \geq 0} P_{V_{1,8}}(2n+1)q^n & = \dfrac{\Theta_4^3}{\Theta_1\Theta_2}.\label{eq:V1 2n+1}
\end{align}
\end{corollary}

\subsubsection*{The $3$-Dissections}

\begin{theorem}
\label{thm:ord2diss3}
We have the following representations
    \begin{align}
    A_2(q) & = \dfrac{A_2(q^9)}{q} - \dfrac{\phi_6(-q^3)}{2} + \dfrac{\Theta_6^{11}\Theta_{18}^2}{2\Theta_3^7\Theta_{12}^4\Theta_{36}} + q\dfrac{\Theta_6^4\Theta_9^2\Theta_{12}}{\Theta_3^5\Theta_{18}} + 2q^2\dfrac{\Theta_6^3\Theta_{12}\Theta_{18}^2}{\Theta_3^4\Theta_9}, 
\label{eq:A2diss3}\\
        B_2(q) & = \dfrac{\psi_6(q^{12})}{q^5} - \dfrac{\phi_6(q^{12})}{q} + \dfrac{\Theta_6^7 \Theta_9^2}{\Theta_3^6 \Theta_{12} \Theta_{18}}
+\dfrac{1}{q}\dfrac{\Theta_6^8 \Theta_{18}^2}{\Theta_3^4 \Theta_{24}^4 \Theta_{36}}
-\dfrac{1}{4q^5}\dfrac{\Theta_9 \Theta_{12}^5 \Theta_{18}^3}{\Theta_3^9 \Theta_6^5 \Theta_{24}^2 \Theta_{36}^2 \Theta_{72}} \nonumber \\
&+\dfrac{1}{4q}\dfrac{\Theta_{12}^4 \Theta_{36}}{\Theta_{24}^9 \Theta_{72}^2}
+\dfrac{3}{4q^2}\dfrac{\Theta_9^7 \Theta_{12}^2 \Theta_{36}^4}{\Theta_3^6 \Theta_{18}^9 \Theta_{24}^2 \Theta_{72}^2},
\label{eq:B2diss3}\\
   \mu_2(q) & = 2\phi_6(q^3) - \dfrac{\mu_2(q^9)}{q} + \dfrac{1}{q}\dfrac{\Theta_6^6\Theta_{18}^2}{\Theta_3\Theta_{12}^5\Theta_{36}} - \dfrac{\Theta_3^4\Theta_9\Theta_{36}}{\Theta_6\Theta_{12}^3\Theta_{18}}-q\dfrac{\Theta_6^7\Theta_{36}^2}{\Theta_3\Theta_{12}^6\Theta_{18}}. 
\label{eq:mu2diss3}
    \end{align}
\end{theorem}

\begin{theorem}\label{thm:ord6diss3}
We have the following representations
    \begin{align}
        \psi_6(q) & =\phi_6(q^9) -\dfrac{\psi_6(q^9)}{q^3}-\dfrac{\Theta_3\Theta_9^4}{\Theta_6^2\Theta_{18}^2}+q\dfrac{\Theta_3^2\Theta_9\Theta_{18}}{\Theta_6^3}-q^2\dfrac{\Theta_3^3\Theta_{18}^4}{\Theta_6^4\Theta_9^2}, 
\label{eq:psi6diss3}\\
        \rho_6(q) & = \dfrac{\psi_6(q^{18})}{q^7}-\dfrac{\phi_6(q^{18})}{q}+\dfrac{\Theta_6^3\Theta_9^4}{\Theta_3^4\Theta_{18}^2}+2q\dfrac{\Theta_6^2\Theta_9\Theta_{18}}{\Theta_3^3}+\dfrac{1}{q}\dfrac{\Theta_6\Theta_9\Theta_{12}\Theta_{18}^4}{\Theta_3^3\Theta_{36}^3}, 
\label{eq:rho6diss3}\\
        \lambda_6(q) & = \dfrac{2}{q}\phi_6(q^{18})-2q^2\rho_6(-q^9)+\dfrac{\Theta_3^5\Theta_9\Theta_{18}}{\Theta_6^6}-q\dfrac{\Theta_3^6\Theta_{18}^4}{\Theta_6^7\Theta_9^2}-3q^2\dfrac{\Theta_3^3\Theta_{18}^5}{\Theta_6^6\Theta_9} - \dfrac{2}{q}\dfrac{\Theta_3^4\Theta_9^2\Theta_{12}^2\Theta_{54}^5}{\Theta_6^6\Theta_{18}^2\Theta_{27}^2\Theta_{108}^2}, 
\label{eq:lambda6diss3}\\
        \psi_{{-}_6}(q) & = \dfrac{\psi_6(q^9)}{2q^3}-\dfrac{\phi_6(q^9)}{2}+\dfrac{\Theta_6^7\Theta_9^7}{2\Theta_3^8\Theta_{18}^5}+q\dfrac{\Theta_6^6\Theta_9^4}{\Theta_3^7\Theta_{18}^2}+2q^2\dfrac{\Theta_6^5\Theta_9\Theta_{18}}{\Theta_3^6}.  
\label{eq:psiminus6diss3}
    \end{align}
\end{theorem}

\begin{theorem}\label{thm:ord8diss3}
We have the following representations
    \begin{align}
        U_{0,8}(q) & = \phi_6(q^3) - \dfrac{U_{0,8}(q^9)}{q}+ \dfrac{1}{q}\dfrac{\Theta_6^2\Theta_{12}\Theta_{36}^5}{\Theta_3\Theta_{18}^2\Theta_{24}^2\Theta_{72}^2} +q\dfrac{\Theta_6\Theta_{12}^3\Theta_{18}\Theta_{72}}{\Theta_3\Theta_{24}^3\Theta_{36}} +q^3\dfrac{\Theta_3^3\Theta_{12}\Theta_{72}^2}{\Theta_6^2\Theta_{24}^2\Theta_{36}},   
\label{eq:U08diss3}\\
        U_{1,8}(q) & = \dfrac{\phi_6(q^3)}{2} + \dfrac{U_{0,8}(q^9)}{2q} - \dfrac{1}{2q} \dfrac{\Theta_6^4\Theta_9\Theta_{24}^2\Theta_{36}^4}{\Theta_{12}^6\Theta_{18}^2\Theta_{72}^2} + \dfrac{\Theta_3^3\Theta_{24}^2\Theta_{36}^5}{2\Theta_{12}^5\Theta_{18}^2\Theta_{72}^2} + q\dfrac{\Theta_6^4\Theta_{24}^3\Theta_{36}^2}{\Theta_3\Theta_{12}^6\Theta_{72}},   
\label{eq:U18diss3}\\
        V_{0,8}(q) & = \dfrac{2}{q^9}\psi_6(q^{24}) 
        - \dfrac{2}{q}\phi_6(q^{24})
        +\dfrac{2}{q}\dfrac{\Theta_6\Theta_{24}^3\Theta_{36}^2}{\Theta_3^2\Theta_{48}^2\Theta_{72}} 
        + \dfrac{\Theta_6^5\Theta_{18}^2\Theta_{24}^4}{\Theta_3^3\Theta_9\Theta_{12}^4\Theta_{48}^2}
        + 2q\dfrac{\Theta_6^3\Theta_9\Theta_{12}\Theta_{36}}{\Theta_3^3\Theta_{18}\Theta_{24}} \nonumber \\
        & + q^3\,\dfrac{\Theta_6^8\Theta_9^2\Theta_{24}\Theta_{36}^2\Theta_{144}}{\Theta_3^4\Theta_{12}^4\Theta_{18}^3\Theta_{48}\Theta_{72}} 
        + 4q^9\,\dfrac{\Theta_6\Theta_{12}^2\Theta_{144}^2}{\Theta_3^2\Theta_{24}\Theta_{72}},    
\label{eq:V08diss3}\\
        V_{1,8}(q) & = -\dfrac{\phi_6(-q^6)}{2} - \dfrac{U_{0,8}(-q^{18})}{2q^2} + \dfrac{1}{2q^2}\dfrac{\Theta_6\Theta_{18}^3\Theta_{24}^5}{\Theta_3^2\Theta_{12}^2\Theta_{36}\Theta_{48}^2\Theta_{72}} + \dfrac{\Theta_9^2\Theta_{12}^7}{2\Theta_3^2\Theta_6^2\Theta_{18}\Theta_{24}^3} \nonumber \\
        & + q^2\dfrac{\Theta_6^4\Theta_9\Theta_{24}\Theta_{36}}{\Theta_3^3\Theta_{12}^2\Theta_{18}}.   
\label{eq:V18diss3}
    \end{align}
\end{theorem}

The following corollaries follow immediately from the preceding theorems.

\begin{corollary}
\label{cor:A2diss3}
We have
    \begin{align}
        \sum_{n \geq 0} P_{A_2}(3n) q^n & =  - \dfrac{\phi_6(-q)}{2} + \dfrac{\Theta_2^{11}\Theta_{6}^2}{2\Theta_1^7\Theta_{4}^4\Theta_{12}},     
\label{eq:A2_3n0}\\
    \sum_{n \geq 0} P_{A_2}(3n + 1) q^n     &= \dfrac{\Theta_{2}^4 \Theta_{3}^2 \Theta_{4}}{\Theta_{1}^5 \Theta_{6}},     \label{eq:A2_3n1}\\
    \sum_{n \geq 0} P_{A_2}(3n + 2) q^n & = \dfrac{A_2(q^3)}{q} + 2\dfrac{\Theta_2^3\Theta_{4}\Theta_{6}^2}{\Theta_1^4\Theta_3}.     \label{eq:A2_3n2}
    \end{align}
\end{corollary}

\begin{corollary}
\label{cor:B2diss3}
We have
    \begin{align}
        \sum_{n \geq 0} P_{B_2}(3n) q^n &= \dfrac{\Theta_{2}^7 \Theta_{3}^2}{\Theta_{1}^6 \Theta_{4} \Theta_{6}},     \label{eq:B3n}\\
    \sum_{n \geq 0} P_{B_2}(3n+1) q^n & = \dfrac{\psi_6(q^{4})}{q^2} + \frac{1}{4q^2} \frac{\Theta_{4}^4\Theta_{12}}{\Theta_{8}^2\Theta_{24}^2}-\frac{1}{4q^2}\frac{\Theta_3\Theta_{4}^9\Theta_{6}^3}{\Theta_1^3\Theta_2\Theta_{8}^5\Theta_{12}^2\Theta_{24}}  \nonumber \\
    &+ \frac{3}{4q} \frac{\Theta_2^7\Theta_3^2\Theta_{12}^4}{\Theta_1^6\Theta_{4}\Theta_{6}\Theta_{8}^2\Theta_{24}^2},     \label{eq:B3n1}\\
    \sum_{n \geq 0} P_{B_2}(3n+2) q^n & = - \frac{\phi_6(q^{4})}{q} +\frac{1}{q}\frac{\Theta_{4}^8\Theta_{6}^2}{\Theta_1^4\Theta_{8}^4\Theta_{12}}.     \label{eq:B3n2}
    \end{align}
\end{corollary}

\begin{corollary}
\label{cor:mu2diss3}           
We have
    \begin{align}
    \sum_{n \geq 0} P_{\mu_2}(3n) q^n & = 2\phi_6(q) - \dfrac{\Theta_2^7\Theta_{12}^2}{\Theta_1\Theta_{4}^6\Theta_{6}},     \label{eq:mu3n}\\
    \sum_{n \geq 0} P_{\mu_2}(3n + 1) q^n     &= -\,\dfrac{\Theta_{2}^7 \Theta_{12}^2}{\Theta_{1} \Theta_{4}^6 \Theta_{6}},     \label{eq:mu2_3n1} \\
    \sum_{n \geq 0} P_{\mu_2}(3n + 2) q^n & = - \dfrac{\mu_2(q^3)}{q} + \dfrac{1}{q}\dfrac{\Theta_2^6\Theta_{6}^2}{\Theta_1\Theta_{4}^5\Theta_{12}}. \label{eq:mu2_3n2}
\end{align}
\end{corollary}

\begin{corollary}
\label{cor:psi6diss3}
We have
    \begin{align}
        \sum_{n \geq 0} P_{\psi_6}(3n) q^n & =\phi_6(q^3) -\dfrac{\psi_6(q^3)}{q}-\dfrac{\Theta_1\Theta_3^4}{\Theta_2^2\Theta_{6}^2},   \label{eq:psi 3n}\\
    \sum_{n \geq 0} P_{\psi_6}(3n+1) q^n &=   \dfrac{\Theta_1^2\Theta_3\Theta_6}{\Theta_2^3},   \label{eq:psi 3n+1}\\
    \sum_{n \geq 0} P_{\psi_6}(3n+2) q^n &=   -\dfrac{\Theta_1^3\Theta_6^4}{\Theta_2^4\Theta_3^2}.   \label{eq:psi 3n+2}
    \end{align}
\end{corollary}

\begin{corollary}
\label{cor:rho6diss3}
We have
    \begin{align}
        \sum_{n \geq 0} P_{\rho_6}(3n) q^n &= \dfrac{\Theta_{2}^3 \Theta_{3}^4}{\Theta_{1}^4 \Theta_{6}^2},     \label{eq:rho 3n}\\
    \sum_{n \geq 0} P_{\rho_6}(3n+1) q^n &= 2\dfrac{\Theta_{2}^2 \Theta_{3} \Theta_{6}}{\Theta_{1}^3 },    \label{eq:rho 3n1}\\
    \sum_{n \geq 0} P_{\rho_6}(3n+2) q^n & = \dfrac{\psi_6(q^{6})}{q^3}-\dfrac{\phi_6(q^{6})}{q}+\dfrac{1}{q}\dfrac{\Theta_2\Theta_3\Theta_{4}\Theta_{6}^4}{\Theta_1^3\Theta_{12}^3}.   \label{eq:rho 3n2}
    \end{align}
\end{corollary}

\begin{corollary}
\label{cor:lambda6diss3}
We have
    \begin{align}
        \sum_{n \geq 0} P_{\lambda_6}(3n) q^n &=   \dfrac{\Theta_1^5\Theta_3\Theta_6}{\Theta_2^6},   \label{eq:lambda 3n}\\
    \sum_{n \geq 0} P_{\lambda_6}(3n+1) q^n &=   -\dfrac{\Theta_1^6\Theta_6^4}{\Theta_2^7\Theta_3^2},   \label{eq:lambda 3n+1}\\
    \sum_{n \geq 0} P_{\lambda_6}(3n+2) q^n & = \dfrac{2}{q}\phi_6(q^{6})-2\rho_6(-q^3)-3\dfrac{\Theta_1^3\Theta_{6}^5}{\Theta_2^6\Theta_3} - \dfrac{2}{q}\dfrac{\Theta_1^4\Theta_3^2\Theta_{4}^2\Theta_{18}^5}{\Theta_2^6\Theta_{6}^2\Theta_{9}^2\Theta_{36}^2}.   \label{eq:lambda 3n+2}
    \end{align}
\end{corollary}

\begin{corollary}
\label{cor:psiminus6diss3}
We have
    \begin{align}
    \sum_{n \ge 0} P_{{\psi_{-}}_6}(3n) q^{n} & = \dfrac{\psi_6(q^3)}{2q}-\dfrac{\phi_6(q^3)}{2}+\dfrac{\Theta_2^7\Theta_3^7}{2\Theta_1^8\Theta_{6}^5}, \label{eq:psiminus 3n}\\
    \sum_{n \ge 0} P_{{\psi_{-}}_6}(3n+1) q^{n}&= \dfrac{\Theta_2^6\Theta_3^4}{\Theta_1^7\Theta_6^2}, \label{eq:psiminus 3n+1}\\
    \sum_{n \ge 0} P_{{\psi_{-}}_6}(3n+2) q^{n}&= 2\dfrac{\Theta_2^5\Theta_3\Theta_6}{\Theta_1^6}. \label{eq:psiminus 3n+2}
\end{align}
\end{corollary}

\begin{corollary}
\label{cor:U08diss3}
We have
    \begin{align}
        \sum_{n \geq 0} P_{U_{0,8}}(3n)q^n & = \phi_6(q) +q\dfrac{\Theta_1^3\Theta_{4}\Theta_{24}^2}{\Theta_2^2\Theta_{8}^2\Theta_{12}},\label{eq:U_0 3n}\\
    \sum_{n \geq 0} P_{U_{0,8}}(3n+1)q^n & = \dfrac{\Theta_2\Theta_4^3\Theta_6\Theta_{24}}{\Theta_1\Theta_8^3\Theta_{12}},\label{eq:U_0 3n+1}\\
    \sum_{n \geq 0} P_{U_{0,8}}(3n+2)q^n & = - \dfrac{U_{0,8}(q^3)}{q}+ \dfrac{1}{q}\dfrac{\Theta_2^2\Theta_{4}\Theta_{12}^5}{\Theta_1\Theta_{6}^2\Theta_{8}^2\Theta_{24}^2}.\label{eq:U_0 3n+2}
    \end{align}
\end{corollary}

\begin{corollary}
\label{cor:U18diss3}        
We have
    \begin{align}
                \sum_{n \geq 0} P_{U_{1,8}}(3n)q^n & = \dfrac{\phi_6(q)}{2} + \dfrac{\Theta_1^3\Theta_{8}^2\Theta_{12}^5}{2\Theta_{4}^5\Theta_{6}^2\Theta_{24}^2},\label{eq:U_1 3n}\\
    \sum_{n \geq 0} P_{U_{1,8}}(3n+1)q^n & = \dfrac{\Theta_2^4\Theta_8^3\Theta_{12}^2}{\Theta_1\Theta_4^6\Theta_{24}},\label{eq:U_1 3n+1}\\
    \sum_{n \geq 0} P_{U_{1,8}}(3n+2)q^n & = \dfrac{U_{0,8}(q^3)}{2q} - \dfrac{1}{2q} \dfrac{\Theta_2^4\Theta_3\Theta_{8}^2\Theta_{12}^4}{\Theta_{4}^6\Theta_{6}^2\Theta_{24}^2}.\label{eq:U_1 3n+2}
    \end{align}
\end{corollary}

\begin{corollary}
\label{cor:V0diss3}
We have
    \begin{align}
        \sum_{n \geq 0} P_{V_{0,8}}(3n)q^n & = \dfrac{2}{q^3}\,\psi_6(q^8) + \dfrac{\Theta_2^5\Theta_6^2\Theta_8^4}{\Theta_1^3\Theta_3\Theta_4^4\Theta_{16}^2} + q\,\dfrac{\Theta_2^8\Theta_3^2\Theta_8\Theta_{12}^2\Theta_{48}}{\Theta_1^4\Theta_4^4\Theta_6^3\Theta_{16}\Theta_{24}} \nonumber \\
        & + 4q^3\,\dfrac{\Theta_2\Theta_4^2\Theta_{48}^2}{\Theta_1^2\Theta_8\Theta_{24}},\label{eq:V0 3n}\\  
    \sum_{n \geq 0} P_{V_{0,8}}(3n+1)q^n & = 2\dfrac{\Theta_2^3\Theta_3\Theta_4\Theta_{12}}{\Theta_1^3\Theta_6\Theta_8},\label{eq:V0 3n+1}\\
    \sum_{n \geq 0} P_{V_{0,8}}(3n+2)q^n & = - \dfrac{2}{q}\phi_6(q^{8}) +\dfrac{2}{q}\dfrac{\Theta_2\Theta_{8}^3\Theta_{12}^2}{\Theta_1^2\Theta_{16}^2\Theta_{24}}.\label{eq:V0 3n+2}
    \end{align}
\end{corollary}

\begin{corollary}
\label{cor:V18diss3}             
For all $n \ge 0$, we have
    \begin{align}
     \sum_{n \geq 0} P_{V_{1,8}}(3n)q^n & = -\dfrac{\phi_6(-q^2)}{2} + \dfrac{\Theta_3^2\Theta_{4}^7}{2\,\Theta_1^2\Theta_2^2\Theta_{6}\Theta_{8}^3},\label{eq:V1 3n}\\
    \sum_{n \geq 0} P_{V_{1,8}}(3n+1)q^n & = - \dfrac{U_{0,8}(-q^{6})}{2q} + \dfrac{1}{2q}\dfrac{\Theta_2\Theta_{6}^3\Theta_{8}^5}{\Theta_1^2\Theta_{4}^2\Theta_{12}\Theta_{16}^2\Theta_{24}},\label{eq:V1 3n+1}\\
    \sum_{n \geq 0} P_{V_{1,8}}(3n+2)q^n & = \dfrac{\Theta_2^4\Theta_3\Theta_8\Theta_{12}}{\Theta_1^3\Theta_4^2\Theta_6}.\label{eq:V1 3n+2}
\end{align}
\end{corollary}
\section{Properties of Theta Functions and Appell-Lerch Sums}
\label{sec:proptm}
We collect some well-known theta function identities from Mortenson's paper \cite[Section~3]{mortenson2024ramanujan}. Certain theta functions can also be expressed in terms of eta products $\Theta_m$:
{\allowdisplaybreaks 
\begin{subequations}
\begin{gather}
\overline{\Theta}_{0,1} =2\overline{\Theta}_{1,4}
=\frac{2\Theta_2^2}{\Theta_1},  \ 
\overline{\Theta}_{1,2}=\frac{\Theta_2^5}{\Theta_1^2\Theta_4^2}, \ 
\Theta_{1,2}=\frac{\Theta_1^2}{\Theta_2},   \ 
\overline{\Theta}_{1,3}=\frac{\Theta_2\Theta_3^2}{\Theta_1\Theta_6},
   \label{eq:thetaeta1}\\
\Theta_{1,4}=\frac{\Theta_1\Theta_4}{\Theta_2},  
\  \Theta_{1,6}=\frac{\Theta_1\Theta_6^2}{\Theta_2\Theta_3},   \ 
\overline{\Theta}_{1,6}
=\frac{\Theta_2^2\Theta_3\Theta_{12}}{\Theta_1\Theta_4\Theta_6}.
\label{eq:thetaeta2}
\end{gather}
\end{subequations}}%

We have 
\begin{align}            
\Theta(x;q)&=\Theta(q/x;q)=-x\Theta(x^{-1};q)
\label{eq:theta1},         \\
\Theta(xq;q)&=-x^{-1}\Theta(x;q)
\label{eq:theta2},         \\
\Theta(x;q)&={\Theta_1}\Theta(x;q^2)\Theta(qx;q^2)/{\Theta_2^2}, 
\label{eq:theta3}          \\
\Theta(z;q)&=\Theta(-z^2q;q^4)-z\Theta(-z^2q^3;q^4),
\label{eq:theta4}
\end{align}   

We will also require the following identities:
\begin{lemma}
\label{lem:etadiss2}
We have the following $2$-dissections for the ordinary partition function and its inverse:
\begin{align}
    \dfrac{1}{\Theta_1} &= \dfrac{\overline{\Theta}_{6,16}}{\Theta_2^2} 
    + q\,\dfrac{\overline{\Theta}_{2,16}}{\Theta_2^2}, 
    \label{eq:eta1diss2}\\
    \Theta_1 &= \dfrac{\Theta_2 \overline{\Theta}_{6,16}}{\Theta_4} 
    + q\,\dfrac{\Theta_2 \overline{\Theta}_{2,16}}{\Theta_4}.
    \label{eq:eta2diss2}
\end{align}
\end{lemma}

\begin{proof}
We begin by writing
\begin{align*}
    \dfrac{1}{\Theta_1}
    &= \dfrac{1}{\Theta_2^2}\,\dfrac{\Theta_2^2}{\Theta_1} \\
    &= \dfrac{1}{\Theta_2^2}\,\overline{\Theta}_{1,4} \\
    &= \dfrac{1}{\Theta_2^2}
       \left( \overline{\Theta}_{6,16}
       + q\,\overline{\Theta}_{2,16} \right),
\end{align*}
where, in the final step, we have used \eqref{eq:theta4},
with $z=-q$ and $q$ replaced by $q^4$.
This establishes \eqn{eta1diss2}.

Replacing $q$ by $-q$ and invoking the identity
\[
(-q;-q)_{\infty} = \dfrac{\Theta_2^3}{\Theta_1 \Theta_4},
\]
in \eqref{eq:eta1diss2}, we immediately deduce \eqref{eq:eta2diss2}.
\end{proof}

\begin{remark}
Oddmund Kolberg \cite{kolberg1957some} obtained the $2$-, $3$-, $5$-, and $7$-dissections of the ordinary partition function. In the case of the $2$-dissection, the even and odd parts each split into two components (see Equations~$(6.1)$ and~$(6.2)$). In our approach, both the even and odd parts are given by single components (see Equation~\eqref{eq:eta1diss2}).
\end{remark}

Next, we recall the definition and several fundamental properties of the Appell--Lerch sum, which will be used throughout the remainder of this paper.

\begin{definition}[{\cite[Definition~1.1]{hickerson2014hecke}}]
\label{def:1p1}               
Let $x, z \in \mathbb{C}^*$ be such that neither $z$ nor $xz$ is an integral power of $q$. 
The Appell--Lerch sum is defined by
\begin{equation}
\label{eq:mxqz}
m(x,q,z) := \dfrac{1}{\Theta(z;q)} 
\sum_{r \in \mathbb{Z}} \dfrac{(-1)^r q^{\binom{r}{2}} z^r}{1 - q^{r-1}xz}.
\end{equation}
\end{definition}

\begin{proposition}[{\cite[Proposition~3.1]{hickerson2014hecke}}]
\label{propo:prop3p1}
For generic $x,z \in \C^\ast$, the Appell--Lerch sum satisfies
\begin{align}
m(x,q,z) &= m(x,q,qz), \label{eq:m_shift_z} \\
m(x,q,z) &= x^{-1} m(x^{-1},q,z^{-1}), \label{eq:m_inversion} \\
m(qx,q,z) &= 1- x\, m(x,q,z), \label{eq:m_qx_relation} \\
m(x,q,z) &= 1 - q^{-1} x\, m(q^{-1}x,q,z), \label{eq:m_q_inverse_relation} \\
m(x,q,z) &= x^{-1} - x^{-1} m(qx,q,z). \label{eq:m_alternative_form}
\end{align}
\end{proposition}

We also require the following difference formula, which provides a useful relation 
between the values of $m(x,q,z)$ at distinct arguments.

\begin{theorem}[{\cite[Theorem~3.3]{hickerson2014hecke}}]
\label{thm:thm3p3}
For generic $x, z_0, z_1 \in \C^\ast$, we have
\begin{equation}\label{eq:m_difference}
m(x,q,z_1) - m(x,q,z_0)
=
\frac{
z_0 \Theta_1^{\,3}\,
\Theta\!\left(z_1/z_0;q\right)\,
\Theta(x z_0 z_1;q)
}{
\Theta(z_0;q)\,
\Theta(z_1;q)\,
\Theta(x z_0;q)\,
\Theta(x z_1;q)
}.
\end{equation}
\end{theorem}

Finally, we recall the decomposition formula for the Appell--Lerch sum, 
which will play a central role in establishing the $2$- and $3$-dissection 
identities developed later in this paper.

\begin{corollary}[{\cite[Corollary~3.6]{hickerson2014hecke}}]
\label{cor:cor3p6}   
Let $n$ be a positive odd integer. For generic $x,z,z'\in\C^\ast$, we have
\begin{align}
m(x,q,z)
&= \sum_{r=0}^{n-1} q^{-\binom{r+1}{2}}(-x)^r\,
m\!\left(q^{\binom{n}{2}-nr}x^n,\;q^{n^2},\;z'\right) \notag\\
&\quad + \dfrac{z'\Theta_n^3}{\Theta(xz;q)\,\Theta(z';q^{n^2})}
\sum_{r=0}^{n-1} q^{r(r-n)/2}(-x)^r z^{\,r-(n-1)/2}\,
\dfrac{\Theta\!\left(q^r x^n z z';\,q^n\right)\,
       \Theta\!\left(q^{nr}z^n/z';\,q^{n^2}\right)}
     {\Theta\!\left(x^n z',\,q^r z;\,q^n\right)} .
\label{eq:m_odd_n}
\end{align}

If $n$ is a positive even integer, then for generic $x,z,z'\in\C^\ast$,
\begin{align}
m(x,q,z)
&= \sum_{r=0}^{n-1} q^{-\binom{r+1}{2}}(-x)^r\,
m\!\left(-q^{\binom{n}{2}-nr}x^n,\;q^{n^2},\;z'\right) \notag\\
&\quad + \dfrac{z'\Theta_n^3}{\Theta(xz;q)\,\Theta(z';q^{n^2})}
\sum_{r=0}^{n-1} q^{r(r-n+1)/2}(-x)^r z^{\,r+1-n/2}\,
\dfrac{\Theta\!\left(-q^{r+n/2}x^n z z';\,q^n\right)\,
       \Theta\!\left(q^{nr}z^n/z';\,q^{n^2}\right)}
     {\Theta\!\left(-q^{n/2}x^n z',\,q^r z;\,q^n\right)} .
\label{eq:m_even_n}
\end{align}
\end{corollary}

\section{The Method}  
\label{sec:meth}

Hickerson and Mortenson \cite{hickerson2014hecke} showed how to use properties of Appell--Lerch sums to prove identities for mock theta functions. We use this idea to derive and prove $m$-dissection identities for mock theta functions. In this process, we also prove identities involving certain theta functions. The first author wrote a MAPLE package \texttt{thetaids}, which implements an algorithm from the theory of modular functions to prove theta function identities. This algorithm is described in a book with Frye \cite{frye2019automatic}. This method was recently used by Mortenson \cite{mortenson2024ramanujan} to prove new mock theta function identities analogous to Ramanujan's 10th order identities.

\subsection{An example}
\label{subsec:eg}

As an example, we show how to derive and prove the $2$-dissection identity in Theorem~\ref{thm:ord2diss2} for the second order mock theta function $A_2(q)$. From \eqref{A2}, we have
$$
A_2(q) = -m(q,q^4,q^2).
$$

We apply Corollary~\ref{cor:cor3p6} with $q\mapsto q^4$, $x=q$, $z=q^2$, $z'=-1$, and $n=2$:
\begin{align*}
A_2(q) &= -m\left(-q^{6}, q^{16}, -1\right)
   +\dfrac{1}{q^3} m\left(-\frac{1}{q^{2}}, q^{16}, -1\right)
   +\dfrac{\Theta \left(q^{8}, q^{24}\right)^{3} \Theta \left(q^{8}, q^{8}\right) \Theta \left(-q^{4}, q^{16}\right)}{\Theta \left(q^{3}, q^{4}\right) \Theta \left(-1, q^{16}\right) \Theta \left(q^{6}, q^{8}\right) \Theta \left(q^{2}, q^{8}\right)}\\
&\quad -q^3\,\dfrac{\Theta \left(q^{8}, q^{24}\right)^{3} \Theta \left(q^{12}, q^{8}\right) \Theta \left(-q^{12}, q^{16}\right)}{\Theta \left(q^{3}, q^{4}\right) \Theta \left(-1, q^{16}\right) \Theta \left(q^{6}, q^{8}\right)^{2}}
\\
&= -\frac{1}{q} m\left(-q^{2}, q^{16}, -1\right)
   -m\left(-q^{6}, q^{16}, -1\right)
   +\frac{\Theta_{4}^{2} \Theta_{8}^{2} \Theta_{16}}
           {2 q \Theta_{1} \Theta_{2} \Theta_{32}^{2}},
\end{align*}
where the last equality follows by applying \eqref{eq:m_inversion} from Proposition~\ref{propo:prop3p1} after some simplification.

Next, we compare the two Appell--Lerch series above with those listed in Section 2. We obtain
$$
\frac{1}{2} S_{1,8}(q^2) - \frac{1}{2} S_{0,8}(q^2)
= -\frac{1}{q} m(-q^2, q^{16}, -1) - m(-q^6, q^{16}, -1)
+ \dfrac{1}{2q}\,\dfrac{\overline{\Theta}_{6,16}\,\Theta_{4,16}^2}
       {\Theta_{2,16}^2} - 
q^2\,\dfrac{\overline{\Theta}_{2,16}\Theta_{4,16}^2}
     {2\, \Theta_{6,16}^2}.
$$

Therefore,
\begin{align*}
A_2(q) = 
\frac{1}{2} S_{1,8}(q^2) - \frac{1}{2} S_{0,8}(q^2)
+ \dfrac{1}{2q}\dfrac{\Theta_4^2\,\Theta_8^2\,\Theta_{16}}{\Theta_1\,\Theta_2\,\Theta_{32}^2}
- \dfrac{1}{2q}\dfrac{\overline{\Theta}_{6,16}\,\Theta^2_{4,16}}{\Theta^2_{2,16}}
+ \dfrac{q^2}{2}\dfrac{\overline{\Theta}_{2,16}\,\Theta^2_{4,16}}{\Theta^2_{6,16}}.
\end{align*}

Using computations with the first author's {\tt qseries} package \cite{Ga1999b}, it appears that the sum of the theta functions above can be written as an eta-product. This leads us to conjecture
\begin{equation}
\label{eq:thetaid1}
\dfrac{1}{2q}\dfrac{\Theta_4^2\,\Theta_8^2\,\Theta_{16}}{\Theta_1\,\Theta_2\,\Theta_{32}^2}
- \dfrac{1}{2q}\dfrac{\overline{\Theta}_{6,16}\,\Theta^2_{4,16}}{\Theta^2_{2,16}}
+ \dfrac{q^2}{2}\dfrac{\overline{\Theta}_{2,16}\,\Theta^2_{4,16}}{\Theta^2_{6,16}}
=
\dfrac{\Theta_{4}^{9}}{2\, \Theta_{1} \Theta_{2}^{3} \Theta_{8}^{4}}.
\end{equation}

If we assume this identity, then \eqref{eq:A2diss2} follows from the $2$-dissection of $1/\Theta_1$ given in Lemma~\ref{lem:etadiss2}. Therefore, it remains to prove \eqref{eq:thetaid1}.

This identity can be reduced to a finite computation using the theory of modular functions. The idea is to rewrite \eqref{eq:thetaid1} as an identity involving generalized eta-products and then apply an algorithm developed by the first author and Frye \cite{frye2019automatic}. This algorithm, described in Appendix \ref{sec:algorithm}, consists of six steps, which we now illustrate. The calculations in these steps can be carried out using the first author's Maple package {\tt thetaids}.

\medskip\noindent
{\bf $\theta$-Step 1.}
We first rewrite \eqref{eq:thetaid1} as
\[
0 = 1 - f_1(\tau) + f_2(\tau) - f_3(\tau),
\]
where
\begin{align*}
f_1(\tau) &:=
\frac{
\Theta_{1}
\Theta_{2}
\Theta_{16}
\overline{\Theta}_{6,16}
\Theta_{32}^{2}
}
{
\Theta_{2,16}^{2}
\Theta_{8}^{4}
},
\\
f_2(\tau) &:=
q^{3}
\frac{
\Theta_{1}
\Theta_{2}
\Theta_{16}
\overline{\Theta}_{2,16}
\Theta_{32}^{2}
}
{
\Theta_{6,16}^{2}
\Theta_{8}^{4}
},
\\
f_3(\tau) &:=
q
\frac{
\Theta_{4}^{7}
\Theta_{32}^{2}
}
{
\Theta_{2}^{2}
\Theta_{8}^{6}
\Theta_{16}
},
\end{align*}
and $q=e^{2\pi i\tau}$. Let
\[
g(\tau) := 1 - f_1(\tau) + f_2(\tau) - f_3(\tau).
\]
Thus, the identity is equivalent to
\[
g(\tau) \equiv 0.
\]

\medskip\noindent
{\bf $\theta$-Step 2.}
We check that each function $f_j(\tau)$ is a modular function on $\Gamma_1(32)$.

\medskip\noindent
{\bf $\theta$-Step 3.}
We find a complete set $\mathcal{S}_{32}$ of inequivalent cusps of $\Gamma_1(32)$ and their cusp widths.

\medskip\noindent
{\bf $\theta$-Step 4.}
We compute the invariant order of each $f_j(\tau)$ at every cusp
$\zeta\in\mathcal{S}_{32}$. This gives the orders
\[
\operatorname{ORD}(f_j,\zeta,\Gamma_1(32)).
\]

\medskip\noindent
{\bf $\theta$-Step 5.}
Using these values, we compute
\[
B=
\sum_{\substack{\zeta\in\mathcal{S}_{32}\\ \zeta\neq\infty}}
\min\Big(
\big\{\operatorname{ORD}(f_j,\zeta,\Gamma_1(32)) : 1\le j\le 3\big\}
\cup\{0\}
\Big) = -24.
\]

\medskip\noindent
{\bf $\theta$-Step 6.}
It follows that the identity holds if and only if
\[
\operatorname{ORD}(g,\infty,\Gamma_1(32))>-B = 24.
\]
Equivalently, we need to check that the $q$-expansion of $g(\tau)$
vanishes through $O(q^{24})$. A Maple computation confirms this
through $O(q^{88})$. Therefore, $g(\tau)\equiv 0$.

This completes the proof.

\subsection{A Strategy for Finding and Proving $m$-Dissections}
\label{subsec:strat}

In this section, we describe a strategy for obtaining the 
$m$-dissection of a mock theta function.

\medskip\noindent
\textbf{$m$-Step 1.}\label{step:m1}
Given a mock theta function, first write it in terms of Appell--Lerch sums 
and possibly theta quotients using Hickerson and Mortenson's catalog
from Section~\subsect{HMcat}.

\medskip\noindent
\textbf{$m$-Step 2.}\label{step:m2}
Apply Corollary~\corol{cor3p6} with $n=m$.  
This expresses the mock theta function in terms of
Appell--Lerch sums of the form $m(x,q^{m^2},z)$ together with theta functions.

\medskip\noindent
\textbf{$m$-Step 3.}\label{step:m3}
Apply Proposition~\propo{prop3p1} to simplify the resulting Appell--Lerch sums,
and use \eqref{eq:theta1}--\eqref{eq:theta2} to simplify the theta functions.

\medskip\noindent
\textbf{$m$-Step 4.}\label{step:m4}
Try to match the resulting Appell--Lerch sums with known mock theta functions
in Hickerson and Mortenson's catalog \cite{hickerson2014hecke},
using Theorem~\thm{thm3p3} if needed.

\medskip\noindent
\textbf{$m$-Step 5.}\label{step:m5}
If successful, this gives an expression for the mock theta function
in terms of other mock theta functions evaluated at $q^m$,
together with a sum of theta quotients.

\medskip\noindent
\textbf{$m$-Step 6.}\label{step:m6}
If possible, determine the $m$-dissection of this sum of theta quotients
in terms of theta functions and eta-products. The functions
{\tt sift}, {\tt jacprodmake}, and {\tt qetamake} from the first
author's {\tt qseries} package are useful for this step.

\medskip\noindent
\textbf{$m$-Step 7.}\label{step:m7}
Prove the theta identity obtained in Step~\ref{step:m6} using
the algorithm described in Appendix \ref{sec:algorithm} and Section \ref{subsec:eg}.

\begin{remark}
It is not always possible in $m$-Step 4 to identify the Appell--Lerch sums
as known mock theta functions. For example, when trying to find the $2$- and
$3$-dissections of $S_{0,8}(q)$, $S_{1,8}(q)$, $T_{0,8}(q)$, and $T_{1,8}(q)$,
we were unable to identify suitable mock theta function representations for
the corresponding Appell--Lerch sums. We leave this as an open problem for
interested readers.
\end{remark}

\medskip

To illustrate this method, we now derive the $3$-dissection of the 
second-order mock theta function $B_2(q)$, namely \eqnref{B2diss3}.

\begin{proof}[Proof of \eqnref{B2diss3}]\hfill

\medskip\noindent
\textbf{$m$-Step 1.}
From \eqref{B2}, we have 
\[
B_2(q)  = -\frac{m(1,q^4,q^3)}{q}.
\]

\medskip\noindent
\textbf{$m$-Step 2.}
Apply Corollary~\corol{cor3p6} with $x=1$, $q \to q^4$, $z=q^3$, $n=3$, and $z'=-1$ to obtain
\begin{align*}
    B_2(q) &= -\dfrac{m(q^{12},q^{36},-1)}{q} 
    + \dfrac{m(1,q^{36},-1)}{q^5} 
    - \dfrac{m(q^{-12},q^{36},-1)}{q^{13}} \\
    &\quad - \dfrac{\Theta_{12}^3\overline{\Theta}_{5,12}\overline{\Theta}_{15,36}}
    {q^5\Theta_{1,4}\Theta_{5,12}\overline{\Theta}_{0,36}\overline{\Theta}_{0,12}} 
    + \dfrac{\Theta_{12}^3\overline{\Theta}_{3,12}\overline{\Theta}_{9,36}}
    {q^4\Theta_{1,4}\Theta_{3,12}\overline{\Theta}_{0,36}\overline{\Theta}_{0,12}} \\
    &\quad + \dfrac{\Theta_{12}^3\overline{\Theta}_{1,12}\overline{\Theta}_{3,36}}
    {q^2\Theta_{1,4}\Theta_{1,12}\overline{\Theta}_{0,36}\overline{\Theta}_{0,12}}.
\end{align*}

\medskip\noindent
\textbf{$m$-Step 3.}
Apply \eqref{eq:m_inversion} of Proposition~\propo{prop3p1} 
to the third term to 
obtain
\begin{align*}
    B_2(q) &= -\dfrac{2}{q}m(q^{12},q^{36},-1) 
    + \dfrac{m(1,q^{36},-1)}{q^5} \\
    &\quad - \dfrac{\Theta_{12}^3\overline{\Theta}_{5,12}\overline{\Theta}_{15,36}}
    {q^5\Theta_{1,4}\Theta_{5,12}\overline{\Theta}_{0,36}\overline{\Theta}_{0,12}} 
    + \dfrac{\Theta_{12}^3\overline{\Theta}_{3,12}\overline{\Theta}_{9,36}}
    {q^4\Theta_{1,4}\Theta_{3,12}\overline{\Theta}_{0,36}\overline{\Theta}_{0,12}} \\
    &\quad + \dfrac{\Theta_{12}^3\overline{\Theta}_{1,12}\overline{\Theta}_{3,36}}
    {q^2\Theta_{1,4}\Theta_{1,12}\overline{\Theta}_{0,36}\overline{\Theta}_{0,12}}.
\end{align*}

\medskip\noindent
\textbf{$m$-Step 4.}
Apply Theorem~\thm{thm3p3}                   to the second term with $x=1$,  $z_1=-1$, $z_0=-q^{12}$ and $q \to q^{36}$ to obtain
\begin{align}
    B_2(q) &= -\dfrac{2}{q}m(q^{12},q^{36},-1)  
    + \dfrac{m(1,q^{36},-q^{12})}{q^5} 
    + \dfrac{\Theta_{36}^3\Theta_{12,36}^2}
    {q^5\overline{\Theta}_{0,36}^2\overline{\Theta}_{12,36}^2} \nonumber \\
    &\quad - \dfrac{\Theta_{12}^3\overline{\Theta}_{5,12}\overline{\Theta}_{15,36}}
    {q^5\Theta_{1,4}\Theta_{5,12}\overline{\Theta}_{0,36}\overline{\Theta}_{0,12}}
    + \dfrac{\Theta_{12}^3\overline{\Theta}_{3,12}\overline{\Theta}_{9,36}}
    {q^4\Theta_{1,4}\Theta_{3,12}\overline{\Theta}_{0,36}\overline{\Theta}_{0,12}} \nonumber \\
    &\quad + \dfrac{\Theta_{12}^3\overline{\Theta}_{1,12}\overline{\Theta}_{3,36}}
    {q^2\Theta_{1,4}\Theta_{1,12}\overline{\Theta}_{0,36}\overline{\Theta}_{0,12}}.
    \label{e0.01}
\end{align}

\medskip\noindent
\textbf{$m$-Step 5.}
Use \eqref{phi6} and \eqref{psi6},
\[
\phi_6(q) = 2m(q,q^3,-1), 
\qquad
\psi_6(q) = m(1,q^3,-q),
\]
and substitute into \eqref{e0.01} to obtain
\begin{align*}
    B_2(q) &= -\dfrac{\phi_6(q^{12})}{q} 
     +\dfrac{\psi_6(q^{12})}{q^5} 
    + \Psi(q),
\end{align*}
where
\begin{align*}
\Psi(q) &=
     \dfrac{\Theta_{36}^3\Theta_{12,36}^2}
    {q^5\overline{\Theta}_{0,36}^2\overline{\Theta}_{12,36}^2} 
    - \dfrac{\Theta_{12}^3\overline{\Theta}_{5,12}\overline{\Theta}_{15,36}}
    {q^5\Theta_{1,4}\Theta_{5,12}\overline{\Theta}_{0,36}\overline{\Theta}_{0,12}} \\
    &\quad + \dfrac{\Theta_{12}^3\overline{\Theta}_{3,12}\overline{\Theta}_{9,36}}
    {q^4\Theta_{1,4}\Theta_{3,12}\overline{\Theta}_{0,36}\overline{\Theta}_{0,12}} 
    + \dfrac{\Theta_{12}^3\overline{\Theta}_{1,12}\overline{\Theta}_{3,36}}
    {q^2\Theta_{1,4}\Theta_{1,12}\overline{\Theta}_{0,36}\overline{\Theta}_{0,12}}.
\end{align*}

\medskip\noindent
\textbf{$m$-Step 6.}
Determine the $3$-dissection of the theta function $\Psi(q)$. 
The elements of the $3$-dissection are {\it not} nice products. However after some experimentation
using the {\tt qseries} package we conjecture that
\begin{align}
\Psi(q) &=\dfrac{\Theta_6^7 \Theta_9^2}{\Theta_3^6 \Theta_{12} \Theta_{18}}
+\dfrac{\Theta_6^8 \Theta_{18}^2}{q\,\Theta_3^4 \Theta_{24}^4 \Theta_{36}}
-\dfrac{\Theta_9 \Theta_{12}^5 \Theta_{18}^3}{4 q^5 \Theta_3^9 \Theta_6^5 \Theta_{24}^2 \Theta_{36}^2 \Theta_{72}} \nonumber \\
&+\dfrac{\Theta_{12}^4 \Theta_{36}}{4 q\,\Theta_{24}^9 \Theta_{72}^2}
+\dfrac{3\,\Theta_9^7 \Theta_{12}^2 \Theta_{36}^4}{4 q^2 \Theta_3^6 \Theta_{18}^9 \Theta_{24}^2 \Theta_{72}^2}.
\label{eq:conj:Psiq}
\end{align}

\medskip\noindent
\textbf{$m$-Step 7.}
Prove \eqn{conj:Psiq} by following $\theta$-Step 1 -- $\theta$-Step 7.
We find that \eqn{conj:Psiq} can be written as an identity for generalized
eta-products on $\Gamma_1(72)$ and that $B=-204$. We verify that the first 204 terms
of the $q$-expansion of both sides agree and check to $O(q^{348})$.
This proves \eqn{conj:Psiq} and we obtain the desired
$3$-dissection of $B_2(q)$.
\end{proof}

\section{$2$-Dissection Proofs}
\label{sec:2disspf}

\subsection{Proof of Theorem \thm{ord2diss2}}

\begin{proof}[Proof of \eqn{A2diss2}]
The proof of the $2$-dissection of $A_2(q)$ is given in Section~\subsect{eg}.
\end{proof}

\begin{proof}[Proof of \eqn{B2diss2}]
Following the strategy described in Section~\subsect{strat}, we derive the desired identity.

At $m$-Step~6, we obtain
$$
B_2(q) = -\frac{U_{0,8}(q^4)}{q^2} + \Psi_{B_2,2}(q),
$$
where
$$
\Psi_{B_2,2}(q) = 
\frac{
\Theta_{2}\Theta_{8}^{4}\Theta_{16}
\left(
\Theta_{1,8}^{2}\overline{\Theta}_{6,16}
+
\Theta_{3,8}^{2}\overline{\Theta}_{2,16}
\right)}
{
2q\Theta_{1}\Theta_{4}^{3}\Theta_{32}^{2}
\Theta_{3,8}\Theta_{1,8}
}.
$$

At $m$-Step~6, we are led to conjecture that
\beq
\Psi_{B_2,2}(q)=
\frac{\Theta_{4}^{5}}{\Theta_{2}^{4}}
+
\frac{\Theta_{8}^{6}\Theta_{16}}
{q\Theta_{2}^{2}\Theta_{4}^{2}\Theta_{32}^{2}}.
\label{eq:conj:PsiB22}
\eeq

We rewrite this as an identity involving generalized eta-products on
$\Gamma_1(32)$ and find that $B=-24$.
We verify that the first $24$ terms of the $q$-expansion on both sides agree,
and additionally check the identity up to $O(q^{88})$.
This proves \eqn{conj:PsiB22}, which gives the desired
$2$-dissection of $B_2(q)$.
\end{proof}

\begin{proof}[Proof of \eqn{mu2diss2}]
We apply the strategy described in Section~\subsect{strat}.

At $m$-Step~5, we obtain
$$
\mu_2(q) =
-\frac{U_{0,8}(q^4)}{q^2}
+
\Psi_{\mu_2,2}(q),
$$
where
$$
\Psi_{\mu_2,2}(q) =
-\frac{2 \Theta_{4}^{2} \Theta_{8}^{3}}
{\Theta_{1} \Theta_{2} \Theta_{16}^{2}}
+
\frac{2 \Theta_{4}^{4} \Theta_{16}^{8}}
{q \Theta_{1} \Theta_{2} \Theta_{8}^{5} \Theta_{32}^{4}}
-
\frac{\Theta_{2}^{8}}
{\Theta_{1}^{3} \Theta_{4}^{4}}
-
\frac{2 q^{2} \Theta_{4}^{2} \Theta_{16}^{2} \overline{\Theta}_{2,16}}
{\Theta_{8}^{2} \Theta_{6,16}^{2}}
-
\frac{2 \Theta_{4}^{2} \Theta_{16}^{2} \overline{\Theta}_{6,16}}
{q \Theta_{8}^{2} \Theta_{2,16}^{2}}.
$$

At $m$-Step~6, we are led to conjecture that
\beq
\Psi_{\mu_2,2}(q)=
-\frac{\Theta_{2}^{8}}
{\Theta_{1}^{3}\Theta_{4}^{4}}.
\label{eq:conj:Psimu22}
\eeq

We rewrite this as an identity involving generalized eta-products on
$\Gamma_1(32)$ and find that $B=-24$.
We verify that the first $24$ terms of the $q$-expansion on both sides agree,
and additionally check the identity up to $O(q^{88})$.

Now, substituting \eqn{theta-inv1-sq} and \eqn{eta1diss2} into \eqn{conj:Psimu22} yields the desired
$2$-dissection of $\mu_2(q)$.
\end{proof}

\subsection{Proof of Theorem \thm{ord6diss2}}

\begin{proof}[Proof of \eqn{rho6diss2}]
Using the strategy described in Section~\subsect{strat}, we derive the desired identity.

At $m$-Step~5, we obtain
$$
\rho_6(q) =
-\frac{U_{0,8}(q^6)}{q}
+
\Psi_{\rho_6,2}(q),
$$
where
$$
\Psi_{\rho_6,2}(q)=
\frac{
\Theta_{2}
\Theta_{3}
\Theta_{12}^{4}
\Theta_{24}
\left(
-q\Theta_{1,12}\Theta_{13,12}\overline{\Theta}_{14,24}
+
\Theta_{7,12}^{2}\overline{\Theta}_{2,24}
\right)}
{
2q\Theta_{1}
\Theta_{6}^{4}
\Theta_{48}^{2}
\Theta_{1,12}
\Theta_{7,12}
}.
$$

At $m$-Step~6, we are led to conjecture that
\beq
\Psi_{\rho_6,2}(q)=
\frac{\Theta_{4}^{3}\Theta_{6}^{2}}
{\Theta_{2}^{3}\Theta_{12}}
+
\frac{\Theta_{4}\Theta_{8}^{2}\Theta_{12}^{3}}
{q\Theta_{2}^{2}\Theta_{6}\Theta_{16}\Theta_{48}}.
\label{eq:conj:Psirho62}
\eeq

We rewrite this as an identity involving generalized eta-products on
$\Gamma_1(48)$ and find that $B=-40$.
We verify that the first $40$ terms of the $q$-expansion on both sides agree,
and additionally check the identity up to $O(q^{136})$.
This proves \eqn{conj:Psirho62}, which establishes the desired
$2$-dissection of $\rho_6(q)$.
\end{proof}

\begin{proof}[Proof of \eqn{lambda6diss2}]
We follow the framework described in Section~\subsect{strat}.

At $m$-Step~5, we obtain
$$
\lambda_6(q)=
\frac{\mu_2(q^6)}{q}
+
\Psi_{\lambda_6,2}(q),
$$
where
$$
\Psi_{\lambda_6,2}(q)=
-\frac{
\Theta_{2}^{2}
\Theta_{8}^{4}
\Theta_{12}^{9}}
{q
\Theta_{4}^{5}
\Theta_{6}^{5}
\Theta_{16}
\Theta_{24}^{2}
\Theta_{48}}
-\frac{
\Theta_{4}^{2}
\Theta_{12}^{2}
\Theta_{16}
\Theta_{24}^{5}}
{q
\Theta_{6}^{3}
\Theta_{8}^{3}
\Theta_{48}^{3}}
+\frac{
\Theta_{1}^{3}
\Theta_{6}^{2}}
{\Theta_{2}^{3}
\Theta_{3}}
+\frac{
\Theta_{12}^{8}}
{q
\Theta_{6}^{3}
\Theta_{24}^{4}}.
$$

At $m$-Step~6, we are led to conjecture that
\beq
\Psi_{\lambda_6,2}(q)=
\frac{\Theta_{4}^{3}\Theta_{6}^{2}}
{\Theta_{2}^{3}\Theta_{12}}
-
\frac{\Theta_{8}^{4}\Theta_{12}^{5}}
{q\Theta_{2}\Theta_{4}^{3}\Theta_{24}^{4}}.
\label{eq:conj:Psilambda62}
\eeq

We rewrite this as an identity involving generalized eta-products on
$\Gamma_1(48)$ and find that $B=-80$.
We verify that the first $80$ terms of the $q$-expansion on both sides agree,
and additionally check the identity up to $O(q^{176})$.
This proves \eqn{conj:Psilambda62}, which yields the desired
$2$-dissection of $\lambda_6(q)$.

Hence, we obtain the identity
\beq
\lambda_6(q)=
\frac{\mu_2(q^6)}{q}
-
\frac{
\Theta_{1}
\Theta_{3}
\Theta_{6}
\Theta_{8}^{2}
\Theta_{12}}
{q
\Theta_{2}^{2}
\Theta_{4}
\Theta_{24}^{2}}.
\label{eq:newlambda6id}
\eeq

Therefore, it remains to determine the $2$-dissection of
$\Theta_{1}\Theta_{3}$. We obtain
\beq
\Theta_{1}\Theta_{3}
=
\frac{
\Theta_{2}\Theta_{8}^{2}\Theta_{12}^{4}}
{\Theta_{4}^{2}\Theta_{6}\Theta_{24}^{2}}
-
\frac{
q\Theta_{4}^{4}\Theta_{6}\Theta_{24}^{2}}
{\Theta_{2}\Theta_{8}^{2}\Theta_{12}^{2}}.
\label{eq:E13adiss2}
\eeq

This identity can be proved using the $\theta$-Step Algorithm.
We rewrite it as an identity involving generalized eta-products on
$\Gamma_1(24)$ and find that $B=-8$.
We verify that the first $8$ terms of the $q$-expansion on both sides agree,
and additionally check the identity up to $O(q^{56})$.
This proves \eqn{E13adiss2}.

Substituting \eqn{E13adiss2} into \eqn{newlambda6id} yields the
$2$-dissection identity \eqn{lambda6diss2}.


\end{proof}

\begin{proof}[Proof of \eqn{sigma6diss2}]
In this case, we make a slight modification to the strategy described in Section~\subsect{strat}. We begin with
$$
\sigma_6(q) = -m(q^2,q^6,q).
$$

At $m$-Step~2, instead of applying Corollary~\corol{cor3p6}, we apply Theorem \ref{thm:thm3p3} to find
$$
\sigma_6(q) =
-m(q^2,q^6,-1)
+
\frac{\Theta_{2}^{4}\Theta_{6}^{5}}
{2\Theta_{1}^{2}\Theta_{3}^{2}\Theta_{4}^{2}\Theta_{12}^{2}}.
$$

Using the catalog in Section~\subsect{HMcat}, this gives
\beq
\sigma_6(q)=
-\frac{\phi_6(q^2)}{2}
+
\frac{\Theta_{2}^{4}\Theta_{6}^{5}}
{2\Theta_{1}^{2}\Theta_{3}^{2}\Theta_{4}^{2}\Theta_{12}^{2}}.
\label{eq:sig6id}
\eeq

Therefore, it remains to determine the $2$-dissection of
$(\Theta_{1}\Theta_{3})^{-1}$. We obtain
\beq
\frac{1}{\Theta_{1}\Theta_{3}}
=
\frac{
\Theta_{8}^{2}\Theta_{12}^{5}}
{\Theta_{2}^{2}\Theta_{4}\Theta_{6}^{4}\Theta_{24}^{2}}
+
\frac{
q\Theta_{4}^{5}\Theta_{24}^{2}}
{\Theta_{2}^{4}\Theta_{6}^{2}\Theta_{8}^{2}\Theta_{12}}.
\label{eq:E13diss2}
\eeq

This identity can be proved using the $\theta$-Step Algorithm.
We rewrite it as an identity involving generalized eta-products on
$\Gamma_1(24)$ and find that $B=-8$.
We verify that the first $8$ terms of the $q$-expansion on both sides agree,
and additionally check the identity up to $O(q^{56})$.
This proves \eqn{E13diss2}.

Substituting \eqn{E13diss2} into \eqn{sig6id} yields the
$2$-dissection of $\sigma_6(q)$.
\end{proof}

\begin{proof}[Proof of \eqn{mu6diss2}]
The argument is similar to that of \eqn{sigma6diss2}, with suitable modifications.

We begin with
$$
\mu_6(q)=
2m(q^2,q^6,-1)
-
\frac{\Theta_{1}^{2}\Theta_{3}^{2}}
{2\Theta_{2}^{2}\Theta_{6}}.
$$

Using the catalog in Section~\subsect{HMcat}, this gives
\beq
\mu_6(q)=
\phi_6(q^2)
-
\frac{\Theta_{1}^{2}\Theta_{3}^{2}}
{2\Theta_{2}^{2}\Theta_{6}}.
\label{eq:mu6id}
\eeq

Substituting \eqn{E13adiss2} into \eqn{mu6id} yields the
$2$-dissection identity \eqn{mu6diss2}.
\end{proof}

\subsection{Proof of Theorem \thm{ord8diss2}}

\begin{proof}[Proof of \eqn{V08diss2}]
We begin with
$$
V_{0,8}(q)
=
-2\frac{m(1,q^8,q)}{q}
-
\frac{\Theta_{2}^{3}\Theta_{4}}
{\Theta_{1}^{2}\Theta_{8}}.
$$

As in the proof of \eqn{sigma6diss2}, we apply Theorem~\thm{thm3p3} to obtain
$$
V_{0,8}(q)
=
-2\frac{m(1,q^8,q^6)}{q}
+
\frac{2\Theta_{4}^{2}\Theta_{8}\Theta_{3,8}}
{\Theta_{2}^{2}\Theta_{1,8}}
-
\frac{\Theta_{2}^{3}\Theta_{4}}
{\Theta_{1}^{2}\Theta_{8}}.
$$

Using the catalog in Section~\subsect{HMcat}, this gives
\beq
V_{0,8}(q)
=
2qB_2(q^2)
+
\frac{2\Theta_{4}^{2}\Theta_{8}\Theta_{3,8}}
{\Theta_{2}^{2}\Theta_{1,8}}
-
\frac{\Theta_{2}^{3}\Theta_{4}}
{\Theta_{1}^{2}\Theta_{8}}.
\label{eq:V08id}
\eeq

We are led to conjecture that
\beq
\frac{2\Theta_{4}^{2}\Theta_{8}\Theta_{3,8}}
{\Theta_{2}^{2}\Theta_{1,8}}
-
\frac{\Theta_{2}^{3}\Theta_{4}}
{\Theta_{1}^{2}\Theta_{8}}
=
\frac{\Theta_{4}^{8}}
{\Theta_{2}^{4}\Theta_{8}^{3}}.
\label{eq:V08etaid}
\eeq

This identity can be proved using the $\theta$-Step Algorithm.
We rewrite it as an identity involving generalized eta-products on
$\Gamma_1(8)$ and find that $B=-1$.
We verify that the first $2$ terms of the $q$-expansion on both sides agree,
and additionally check the identity up to $O(q^{17})$.
This proves \eqn{V08etaid}. Substituting this into \eqn{V08id} gives the
$2$-dissection identity \eqn{V08diss2}.
\end{proof}

\begin{proof}[Proof of \eqn{V18diss2}]
The proof is analogous to that of \eqn{V08diss2}, but the computation is shorter.
We obtain
\begin{align*}
V_{1,8}(q)
&= -m(q^2,q^8,q) \\
&= -m(q^2,q^8,q^4)
+ q\frac{\Theta_{8}^{3}}{\Theta_{2}\Theta_{4}} \\
&= A_2(q^2)
+ q\frac{\Theta_{8}^{3}}{\Theta_{2}\Theta_{4}},
\end{align*}
which immediately gives the $2$-dissection identity \eqn{V18diss2}.
\end{proof}
\section{$3$-Dissection Proofs}
\label{sec:3disspf}

\subsection{Proof of Theorem \thm{ord2diss3}}

\begin{proof}[Proof of \eqn{A2diss3}]
We begin with
$$
A_2(q) = -m(q,q^4,q^2).
$$

Applying Corollary~\corol{cor3p6} with $n=3$ and $z'=-1$, together with
\eqref{eq:m_inversion} from Proposition~\propo{prop3p1}, and simplifying, we obtain
\begin{align*}
A_2(q)
&=
-\frac{m(q^9,q^{36},-1)}{q}
-
m(q^{15},q^{36},-1)
+
\frac{m(q^3,q^{36},-1)}{q^2}
\\
&\qquad
+
\frac{
\Theta_{3}\Theta_{12}^{3}\Theta_{18}\overline{\Theta}_{5,12}}
{2q^{2}\Theta_{1}\Theta_{6}^{2}\Theta_{24}\Theta_{72}}
-
\frac{
\Theta_{2}\Theta_{12}^{4}\Theta_{36}^{6}}
{2q^{3}\Theta_{1}\Theta_{4}\Theta_{6}^{2}\Theta_{18}^{2}\Theta_{72}^{4}}
+
\frac{
\Theta_{3}\Theta_{12}^{3}\Theta_{18}\overline{\Theta}_{1,12}}
{2q\Theta_{1}\Theta_{6}^{2}\Theta_{24}\Theta_{72}}
\\
&=
-\frac{m(q^9,q^{36},-1)}{q}
-
m(q^{15},q^{36},-1)
+
\frac{m(q^3,q^{36},-1)}{q^2}
\\
&\qquad
-
\frac{
\Theta_{2}\Theta_{12}^{4}\Theta_{36}^{6}}
{2q^{3}\Theta_{1}\Theta_{4}\Theta_{6}^{2}\Theta_{18}^{2}\Theta_{72}^{4}}
+
\frac{
\Theta_{2}\Theta_{3}^{3}\Theta_{12}^{3}\Theta_{18}}
{2q^{2}\Theta_{1}^{2}\Theta_{6}^{3}\Theta_{24}\Theta_{72}},
\end{align*}
since
$$
q\overline{\Theta}_{1,12}
+
\overline{\Theta}_{5,12}
=
\overline{\Theta}_{1,3}
=
\frac{\Theta_{2}\Theta_{3}^{2}}
{\Theta_{1}\Theta_{6}},
$$
by \eqn{theta4} and \eqn{thetaeta1}.

Next, we apply Theorem~\thm{thm3p3} to the first Appell--Lerch sum to obtain
\begin{align*}
A_2(q)
&=
-\frac{m(q^9,q^{36},q^{18})}{q}
-
m(q^{15},q^{36},-1)
+
\frac{m(q^3,q^{36},-1)}{q^2}
\\
&\qquad
-
\frac{
\Theta_{36}^{9}}
{2q\Theta_{9}\Theta_{18}^{3}\Theta_{72}^{4}}
-
\frac{
\Theta_{2}\Theta_{12}^{4}\Theta_{36}^{6}}
{2q^{3}\Theta_{1}\Theta_{4}\Theta_{6}^{2}\Theta_{18}^{2}\Theta_{72}^{4}}
+
\frac{
\Theta_{2}\Theta_{3}^{3}\Theta_{12}^{3}\Theta_{18}}
{2q^{2}\Theta_{1}^{2}\Theta_{6}^{3}\Theta_{24}\Theta_{72}}.
\end{align*}

Next, we apply Corollary~\corol{cor3p6} with $n=2$, $z'=-1$, and $q\mapsto -q^3$ to
$\phi_6(q)=2m(q,q^3,-1)$, together with \eqref{eq:m_inversion} from
Proposition~\propo{prop3p1}. After simplifying, we obtain
\begin{align*}
\phi_6(-q^3)
&=
-\frac{2m(q^3,q^{36},-1)}{q^3}
+
2m(q^{15},q^{36},-1)
-
\frac{\Theta_{3}\Theta_{18}^{2}}
{\Theta_{6}\Theta_{36}}
+
\frac{
\Theta_{12}\Theta_{36}^{6}}
{q^{3}\Theta_{6}\Theta_{9}\Theta_{72}^{4}}.
\end{align*}

Combining these expressions gives
$$
A_2(q)
=
\frac{A_2(q^9)}{q}
-
\frac{1}{2}\phi_6(-q^3)
+
\Psi_{A_2,3}(q),
$$
where
$$
\Psi_{A_2,3}(q)=
-\frac{
\Theta_{36}^{9}}
{2q\Theta_{9}\Theta_{18}^{3}\Theta_{72}^{4}}
-
\frac{
\Theta_{2}\Theta_{12}^{4}\Theta_{36}^{6}}
{2q^{3}\Theta_{1}\Theta_{4}\Theta_{6}^{2}\Theta_{18}^{2}\Theta_{72}^{4}}
+
\frac{
\Theta_{2}\Theta_{3}^{3}\Theta_{12}^{3}\Theta_{18}}
{2q^{2}\Theta_{1}^{2}\Theta_{6}^{3}\Theta_{24}\Theta_{72}}
-
\frac{
\Theta_{3}\Theta_{18}^{2}}
{2\Theta_{6}\Theta_{36}}
+
\frac{
\Theta_{12}\Theta_{36}^{6}}
{2q^{3}\Theta_{6}\Theta_{9}\Theta_{72}^{4}}.
$$

We are led to conjecture that
\beq
\Psi_{A_2,3}(q)=
\frac{
\Theta_{6}^{11}\Theta_{18}^{2}}
{2\Theta_{3}^{7}\Theta_{12}^{4}\Theta_{36}}
+
\frac{
q\Theta_{6}^{4}\Theta_{9}^{2}\Theta_{12}}
{\Theta_{3}^{5}\Theta_{18}}
+
\frac{
2q^{2}\Theta_{6}^{3}\Theta_{12}\Theta_{18}^{2}}
{\Theta_{3}^{4}\Theta_{9}}.
\label{eq:conj:PsiA23}
\eeq

We rewrite this as an identity involving generalized eta-products on
$\Gamma_1(72)$ and calculate that $B=-228$.
We verify that the first $229$ terms of the $q$-expansion on both sides agree,
and additionally check the identity up to $O(q^{372})$.
This proves \eqn{conj:PsiA23}, which establishes \eqn{A2diss3},
the desired $3$-dissection of $A_2(q)$.
\end{proof}


\begin{proof}[Proof of \eqn{B2diss3}]
The proof of the $3$-dissection of $B_2(q)$ is given in Section~\subsect{strat}.
\end{proof}

\begin{proof}[Proof of \eqn{mu2diss3}]
We begin with
$$
\mu_2(q)=
4m(-q,q^4,-1)
-
\frac{\Theta_{2}^{8}}
{\Theta_{1}^{3}\Theta_{4}^{4}}.
$$

Applying Corollary~\corol{cor3p6} with $n=3$ and $z'=-1$, together with
\eqref{eq:m_inversion} from Proposition~\propo{prop3p1}, and simplifying, we obtain
\begin{align*}
\mu_2(q)
&=
-\frac{4m(-q^9,q^{36},-1)}{q}
+
4m(-q^{15},q^{36},-1)
+
\frac{4m(-q^3,q^{36},-1)}{q^3}
\\
&\qquad
-
\frac{
2\Theta_{2}\Theta_{6}\Theta_{12}\Theta_{24}\Theta_{36}
\overline{\Theta}_{5,12}}
{q^{3}\Theta_{1}\Theta_{3}\Theta_{8}\Theta_{72}^{2}}
+
\frac{
2\Theta_{2}\Theta_{6}\Theta_{12}\Theta_{24}\Theta_{36}
\overline{\Theta}_{1,12}}
{q^{2}\Theta_{1}\Theta_{3}\Theta_{8}\Theta_{72}^{2}}
-
\frac{\Theta_{2}^{8}}
{\Theta_{1}^{3}\Theta_{4}^{4}}.
\end{align*}

Next, we apply Corollary~\corol{cor3p6} with $n=2$, $z'=-1$, and
$q\mapsto q^3$ to $\phi_6(q)=2m(q,q^3,-1)$, together with
\eqref{eq:m_inversion} from Proposition~\propo{prop3p1}. After simplifying, we obtain
\begin{align*}
\phi_6(q^3)
&=
\dfrac{2m(-q^3,q^{36},-1)}{q^3}
+2m(-q^{15},q^{36},-1)
-\dfrac{2\Theta_{18,54}^3\overline{\Theta}_{3,18}}
{\overline{\Theta}_{3,9}\Theta_{3,18}\overline{\Theta}_{0,18}}
\\
&\quad
-\dfrac{2}{q^3}
\dfrac{\Theta_{18,54}^{3}\,\overline{\Theta}_{6,18}\,\overline{\Theta}_{18,36}}
{\overline{\Theta}_{3,9}\,\overline{\Theta}_{0,36}\,\Theta_{3,18}\,\overline{\Theta}_{9,18}}.
\end{align*}

Combining these expressions gives
\begin{align*}
\mu_2(q)
&=
2\phi_6(q^3)
- \frac{\mu_2(q^9)}{q}
+\Psi_{\mu_2,3}(q).
\end{align*}

where
\begin{align*}
\Psi_{\mu_2,3}(q)
&=\dfrac{4}{q^3}\Biggl(
   \dfrac{\Theta_{18}^3\,\overline{\Theta}_{6,18}\,\overline{\Theta}_{18,36}}
        {\Theta_{3,18}\,\overline{\Theta}_{3,9}\,\overline{\Theta}_{0,36}\,\overline{\Theta}_{9,18}}
  -\dfrac{\Theta_{12}^4\,\overline{\Theta}_{5,12}}
        {\Theta_{1,4}\,\Theta_{3,12}\,\overline{\Theta}_{0,36}\,\overline{\Theta}_{4,12}}
 \Biggr)
 +\dfrac{4}{q^2}\,
  \dfrac{\Theta_{12}^4\,\overline{\Theta}_{1,12}}
        {\Theta_{1,4}\,\Theta_{3,12}\,\overline{\Theta}_{0,36}\,\overline{\Theta}_{4,12}}
 \nonumber\\
&\quad
 +4\,\dfrac{\Theta_{18}^3\,\overline{\Theta}_{3,18}}
          {\Theta_{3,18}\,\overline{\Theta}_{3,9}\,\overline{\Theta}_{0,18}}
 -\frac{\Theta_{2}^{8}}
{\Theta_{1}^{3}\Theta_{4}^{4}}.
\end{align*}

We conjecture that
\beq
\Psi_{\mu_2,3}(q)
=
\dfrac{1}{q}\dfrac{\Theta_6^6\Theta_{18}^2}{\Theta_3\Theta_{12}^5\Theta_{36}} - \dfrac{\Theta_3^4\Theta_9\Theta_{36}}{\Theta_6\Theta_{12}^3\Theta_{18}}-q\dfrac{\Theta_6^7\Theta_{36}^2}{\Theta_3\Theta_{12}^6\Theta_{18}}.
\label{eq:conj:Psimu23}
\eeq

We rewrite this as an identity involving generalized eta-products on
$\Gamma_1(72)$ and compute that $B=-240$.
We verify that the first $241$ terms of the $q$-expansions on both sides agree,
and we also check the identity up to $O(q^{384})$.

This proves \eqn{conj:Psimu23}, which establishes \eqn{mu2diss3}, giving the required $3$-dissection of $\mu_2(q)$.
\end{proof}


\begin{proof}[Proof of \eqn{psi6diss3}]
We begin with
$$
\psi_6(q)=m(1,q^3,-q).
$$

Applying Corollary~\corol{cor3p6} with $n=3$ and $z'=-1$, together with
\eqref{eq:m_inversion} from Proposition~\propo{prop3p1}, and simplifying, we get
\begin{align*}
\psi_6(q)
&=2m(q^9,q^{27},-1)-\dfrac{m(1,q^{27},-1)}{q^3}
+\dfrac{\Theta_9^3\Theta_{1,9}\Theta_{3,27}}
{q\,\overline{\Theta}_{1,3}\overline{\Theta}_{0,27}\overline{\Theta}_{0,9}\overline{\Theta}_{1,9}}\\
&\quad
+\dfrac{\Theta_9^3\Theta_{4,9}\Theta_{12,27}}
{q^3\overline{\Theta}_{1,3}\overline{\Theta}_{0,27}\overline{\Theta}_{0,9}\overline{\Theta}_{4,9}}
+\dfrac{\Theta_9^3\Theta_{2,9}\Theta_{6,27}}
{q^2\,\overline{\Theta}_{1,3}\overline{\Theta}_{0,27}\overline{\Theta}_{0,9}\overline{\Theta}_{2,9}}.
\end{align*}

Next, we apply Theorem~\thm{thm3p3} to the second Appell--Lerch sum with $z_0=-q^9$, together with \eqn{theta1} and \eqn{theta2}, and simplify to obtain
\begin{align*}
\psi_6(q)
&=2m(q^9,q^{27},-1)-\dfrac{m(1,q^{27},-q^9)}{q^3}
-\dfrac{1}{q^3}
\dfrac{\Theta_{27,81}^3\,\Theta_{9,27}^2}
{\overline{\Theta}_{9,27}^2\,\overline{\Theta}_{1,27}^2}
\\
&\quad
-2\dfrac{\Theta_{27,81}^3\,\Theta_{1,27}\Theta_{9,27}}
{\overline{\Theta}_{1,27}^2\,\overline{\Theta}_{9,27}^2}
+\dfrac{1}{q}
\dfrac{\Theta_{9,27}^3\,\Theta_{1,9}\Theta_{3,27}}
{\overline{\Theta}_{1,3}\,\overline{\Theta}_{1,27}\,\overline{\Theta}_{1,9}\,\overline{\Theta}_{1,9}}
\nonumber\\
&\quad
+\dfrac{1}{q^3}
\dfrac{\Theta_{9,27}^3\,\Theta_{4,9}\Theta_{12,27}}
{\overline{\Theta}_{1,3}\,\overline{\Theta}_{1,27}\,\overline{\Theta}_{1,9}\,\overline{\Theta}_{4,9}}
+\dfrac{1}{q^2}
\dfrac{\Theta_{9,27}^3\,\Theta_{2,9}\Theta_{6,27}}
{\overline{\Theta}_{1,3}\,\overline{\Theta}_{1,27}\,\overline{\Theta}_{1,9}\,\overline{\Theta}_{2,9}}.
\end{align*}

Using the catalog in Section~\subsect{HMcat}, we obtain
\begin{align*}
\psi_6(q)
=
\phi_6(q^9)-\dfrac{\psi_6(q^9)}{q^3}
-\Psi_{\psi_6,3}(q).
\end{align*}

where
\begin{align*}
\Psi_{\psi_6,3}(q)
&=
-\dfrac{1}{q^3}
\dfrac{\Theta_{27,81}^3\,\Theta_{9,27}^2}
{\overline{\Theta}_{9,27}^2\,\overline{\Theta}_{1,27}^2}
-2\dfrac{\Theta_{27,81}^3\,\Theta_{1,27}\Theta_{9,27}}
{\overline{\Theta}_{1,27}^2\,\overline{\Theta}_{9,27}^2}
+\dfrac{1}{q}
\dfrac{\Theta_{9,27}^3\,\Theta_{1,9}\Theta_{3,27}}
{\overline{\Theta}_{1,3}\,\overline{\Theta}_{1,27}\,\overline{\Theta}_{1,9}\,\overline{\Theta}_{1,9}}
\nonumber\\
&\quad
+\dfrac{1}{q^3}
\dfrac{\Theta_{9,27}^3\,\Theta_{4,9}\Theta_{12,27}}
{\overline{\Theta}_{1,3}\,\overline{\Theta}_{1,27}\,\overline{\Theta}_{1,9}\,\overline{\Theta}_{4,9}}
+\dfrac{1}{q^2}
\dfrac{\Theta_{9,27}^3\,\Theta_{2,9}\Theta_{6,27}}
{\overline{\Theta}_{1,3}\,\overline{\Theta}_{1,27}\,\overline{\Theta}_{1,9}\,\overline{\Theta}_{2,9}}.
\end{align*}

We conjecture that
\beq
\Psi_{\psi_6,3}(q)=
\dfrac{\Theta_6^7\,\Theta_9^7}{2\,\Theta_3^8\,\Theta_{18}^5}
+q\,\dfrac{\Theta_6^6\,\Theta_9^4}{\Theta_3^7\,\Theta_{18}^2}
+2q^2\,\dfrac{\Theta_6^5\,\Theta_9\,\Theta_{18}}{\Theta_3^6}.
\label{eq:conj:Psipsi63}
\eeq

We rewrite this as an identity involving generalized eta-products on
$\Gamma_1(54)$ and compute that $B=-99$.
We verify that the first $100$ terms of the $q$-expansions on both sides agree, and we also check the identity up to $O(q^{207})$.

This proves \eqn{conj:Psipsi63}, which establishes \eqn{psi6diss3}, giving the required $3$-dissection of $\psi_6(q)$.
\end{proof}

\begin{proof}[Proof of \eqn{rho6diss3}]
We begin with
$$
\rho_6(q)=-\dfrac{m(1,q^6,q)}{q}.
$$

Applying Corollary~\corol{cor3p6} with $n=3$ and $z'=-1$, together with
\eqref{eq:m_inversion} from Proposition~\propo{prop3p1}, and simplifying, we get
\begin{align*}
\rho_6(q)
&=-\dfrac{m(q^{18},q^{54},-1)}{q}
+\dfrac{m(1,q^{54},-1)}{q^7}
+\dfrac{\Theta_{18,54}^3\,\overline{\Theta}_{1,18}\,\overline{\Theta}_{3,54}}
{q^2\,\Theta_{1,6}\,\overline{\Theta}_{0,54}\,\overline{\Theta}_{0,18}\,\Theta_{1,18}}
\\
&\quad
-\dfrac{\Theta_{18,54}^3\,\overline{\Theta}_{7,18}\,\overline{\Theta}_{21,54}}
{q^7\,\Theta_{1,6}\,\overline{\Theta}_{0,54}\,\overline{\Theta}_{0,18}\,\Theta_{7,18}}
+\dfrac{\Theta_{18,54}^3\,\overline{\Theta}_{5,18}\,\overline{\Theta}_{15,54}}
{q^6\,\Theta_{1,6}\,\overline{\Theta}_{0,54}\,\overline{\Theta}_{0,18}\,\Theta_{5,18}}.
\end{align*}

Next, we apply Theorem~\thm{thm3p3} to the second Appell--Lerch sum with $z_0=-q^{18}$, together with \eqn{theta1} and \eqn{theta2}, and simplify to obtain
\begin{align*}
\rho_6(q)
&=\dfrac{m(1,q^{54},-q^{18})}{q^7}
-\dfrac{2m(q^{18},q^{54},-1)}{q}
+\dfrac{\Theta_{54,162}^3\,\Theta_{18,54}^2}
{q^7\,\overline{\Theta}_{18,54}^2\,\overline{\Theta}_{1,54}^2}
\\
&\quad
+\dfrac{\Theta_{18,54}^3\,\overline{\Theta}_{1,18}\,\overline{\Theta}_{3,54}}
{q^2\,\Theta_{1,6}\,\overline{\Theta}_{0,54}\,\overline{\Theta}_{0,18}\,\Theta_{1,18}}
\nonumber\\
&\quad
-\dfrac{\Theta_{18,54}^3\,\overline{\Theta}_{7,18}\,\overline{\Theta}_{21,54}}
{q^7\,\Theta_{1,6}\,\overline{\Theta}_{0,54}\,\overline{\Theta}_{0,18}\,\Theta_{7,18}}
+\dfrac{\Theta_{18,54}^3\,\overline{\Theta}_{5,18}\,\overline{\Theta}_{15,54}}
{q^6\,\Theta_{1,6}\,\overline{\Theta}_{0,54}\,\overline{\Theta}_{0,18}\,\Theta_{5,18}}.
\end{align*}

Using the catalog in Section~\subsect{HMcat}, we obtain
\begin{align*}
\rho_6(q)
=
\dfrac{\psi_6(q^{18})}{q^7}
-\dfrac{\phi_6(q^{18})}{q}
+\Psi_{\rho_6,3}(q).
\end{align*}

where
\begin{align*}
\Psi_{\rho_6,3}(q)
&=
\dfrac{\Theta_{54,162}^3\,\Theta_{18,54}^2}
{q^7\,\overline{\Theta}_{18,54}^2\,\overline{\Theta}_{1,54}^2}
+\dfrac{\Theta_{18,54}^3\,\overline{\Theta}_{1,18}\,\overline{\Theta}_{3,54}}
{q^2\,\Theta_{1,6}\,\overline{\Theta}_{0,54}\,\overline{\Theta}_{0,18}\,\Theta_{1,18}}
\nonumber\\
&\quad
-\dfrac{\Theta_{18,54}^3\,\overline{\Theta}_{7,18}\,\overline{\Theta}_{21,54}}
{q^7\,\Theta_{1,6}\,\overline{\Theta}_{0,54}\,\overline{\Theta}_{0,18}\,\Theta_{7,18}}
+\dfrac{\Theta_{18,54}^3\,\overline{\Theta}_{5,18}\,\overline{\Theta}_{15,54}}
{q^6\,\Theta_{1,6}\,\overline{\Theta}_{0,54}\,\overline{\Theta}_{0,18}\,\Theta_{5,18}}.
\end{align*}

We conjecture that
\beq
\Psi_{\rho_6,3}(q)=
\dfrac{\Theta_6^3\Theta_9^4}{\Theta_3^4\Theta_{18}^2}
+2q\dfrac{\Theta_6^2\Theta_9\Theta_{18}}{\Theta_3^3}
+\dfrac{1}{q}\dfrac{\Theta_6\Theta_9\Theta_{12}\Theta_{18}^4}
{\Theta_3^3\Theta_{36}^3}.
\label{eq:conj:Psirho63}
\eeq

We rewrite this as an identity involving generalized eta-products on
$\Gamma_1(108)$ and compute that $B=-396$.
We verify that the first $397$ terms of the $q$-expansions on both sides agree, and we also check the identity up to $O(q^{612})$.

This proves \eqn{conj:Psirho63}, which establishes \eqn{rho6diss3}, giving the required $3$-dissection of $\rho_6(q)$.
\end{proof}

\begin{proof}[Proof of \eqn{lambda6diss3}]
We begin with
$$
\lambda_6(q)=\dfrac{2\, m(1,q^6,-q^2)}{q} + \dfrac{\Theta_1\Theta_3\Theta_{12}}{\Theta_4\Theta_6}.
$$

Applying Corollary~\corol{cor3p6} with $n=3$ and $z'=-1$, together with
\eqref{eq:m_inversion} from Proposition~\propo{prop3p1}, and simplifying, we get
\begin{align*}
\lambda_6(q)
&= \dfrac{4m(q^{18},q^{54},-1)}{q} - \dfrac{2m(1,q^{54},-1)}{q^7} + \dfrac{2\Theta_{18,54}^3 \Theta_{2,18} \Theta_{6,54}}{q^3 \overline{\Theta}_{2,6}\overline{\Theta}_{0,54}\overline{\Theta}_{0,18}\overline{\Theta}_{2,18}} + \dfrac{2 \Theta_{18,54}^3 \Theta_{8,18} \Theta_{24,54}}{q^7 \overline{\Theta}_{2,6}\overline{\Theta}_{0,54}\overline{\Theta}_{0,18}\overline{\Theta}_{8,18}}\\
& + \dfrac{2\Theta_{18,54}^3\Theta_{4,18}\Theta_{12,54}}{q^5 \overline{\Theta}_{2,6}\overline{\Theta}_{0,54}\overline{\Theta}_{0,18}\overline{\Theta}_{4,18}} + \dfrac{\Theta_1\Theta_3\Theta_{12}}{\Theta_4\Theta_6}.
\end{align*}

Next, we apply Theorem~\thm{thm3p3} to the second Appell--Lerch sum with $z_0=-q^{9}$, together with \eqn{theta1} and \eqn{theta2}, and simplify to obtain
\begin{align*}
\lambda_6(q)
&= \dfrac{4m(q^{18},q^{54},-1)}{q} - \dfrac{2 m(1,q^{54},-q^9)}{q^7} - \dfrac{2 \Theta_{54,162}^3 \Theta_{9,54}^2}{q^7 \overline{\Theta}_{9,54}^2\overline{\Theta}_{0,54}^2} + \dfrac{2\Theta_{18,54}^3 \Theta_{2,18} \Theta_{6,54}}{q^3 \overline{\Theta}_{2,6}\overline{\Theta}_{0,54}\overline{\Theta}_{0,18}\overline{\Theta}_{2,18}} \\
& + \dfrac{2 \Theta_{18,54}^3 \Theta_{8,18} \Theta_{24,54}}{q^7 \overline{\Theta}_{2,6}\overline{\Theta}_{0,54}\overline{\Theta}_{0,18}\overline{\Theta}_{8,18}} + \dfrac{2\Theta_{18,54}^3\Theta_{4,18}\Theta_{12,54}}{q^5 \overline{\Theta}_{2,6}\overline{\Theta}_{0,54}\overline{\Theta}_{0,18}\overline{\Theta}_{4,18}} + \dfrac{\Theta_1\Theta_3\Theta_{12}}{\Theta_4\Theta_6}.
\end{align*}

Using the catalog in Section~\subsect{HMcat}, we obtain
\begin{align*}
\lambda_6(q)
=\dfrac{2 \phi_6(q^{18})}{q} - 2q^2\rho_6(-q^9)+\Psi_{\lambda_6,3}(q).
\end{align*}

where
\begin{align*}
\Psi_{\lambda_6,3}(q)
&=- \dfrac{2 \Theta_{54,162}^3 \Theta_{9,54}^2}{q^7 \overline{\Theta}_{9,54}^2\overline{\Theta}_{0,54}^2} + \dfrac{2\Theta_{18,54}^3 \Theta_{2,18} \Theta_{6,54}}{q^3 \overline{\Theta}_{2,6}\overline{\Theta}_{0,54}\overline{\Theta}_{0,18}\overline{\Theta}_{2,18}} + \dfrac{2 \Theta_{18,54}^3 \Theta_{8,18} \Theta_{24,54}}{q^7 \overline{\Theta}_{2,6}\overline{\Theta}_{0,54}\overline{\Theta}_{0,18}\overline{\Theta}_{8,18}}\\
& + \dfrac{2\Theta_{18,54}^3\Theta_{4,18}\Theta_{12,54}}{q^5 \overline{\Theta}_{2,6}\overline{\Theta}_{0,54}\overline{\Theta}_{0,18}\overline{\Theta}_{4,18}} + \dfrac{\Theta_1\Theta_3\Theta_{12}}{\Theta_4\Theta_6}.
\end{align*}

We conjecture that
\beq
\Psi_{\lambda_6,3}(q)=
\dfrac{\Theta_3^5\,\Theta_9\,\Theta_{18}}{\Theta_6^6}
-q\,\dfrac{\Theta_3^6\,\Theta_{18}^4}{\Theta_6^7\,\Theta_9^2}
-3q^2\,\dfrac{\Theta_3^3\,\Theta_{18}^5}{\Theta_6^6\,\Theta_9}
-\dfrac{2}{q}\dfrac{\Theta_3^4\,\Theta_9^2\,\Theta_{12}^2\,\Theta_{54}^5}
        {\Theta_6^6\,\Theta_{18}^2\,\Theta_{27}^2\,\Theta_{108}^2}
\label{eq:conj:Psilambda63}
\eeq

We rewrite this as an identity involving generalized eta-products on
$\Gamma_1(108)$ and compute that $B=-558$.
We verify that the first $559$ terms of the $q$-expansions on both sides agree, and we also check the identity up to $O(q^{774})$.

This proves \eqn{conj:Psilambda63}, which establishes \eqn{lambda6diss3}, giving the required $3$-dissection of $\lambda_6(q)$.
\end{proof}


\begin{proof}[Proof of \eqn{psiminus6diss3}]
We begin with
$$
\psi_{{-}_6}(q)
=
-\dfrac{m(1,q^3,q)}{2}
+\dfrac{q\,\Theta_6^3}{2\Theta_1\Theta_2}.
$$

Applying Corollary~\corol{cor3p6} with $n=3$ and $z'=-1$, together with
\eqref{eq:m_inversion} from Proposition~\propo{prop3p1}, and simplifying, we get
\begin{align*}
\psi_{{-}_6}(q)
&=
-m(q^9,q^{27},-1)
+\dfrac{m(1,q^{27},-1)}{2q^3}
+\dfrac{\Theta_{9,27}^3\overline{\Theta}_{1,9}\overline{\Theta}_{3,27}}
{2q\,\Theta_{1,3}\overline{\Theta}_{0,27}\overline{\Theta}_{0,9}\Theta_{1,9}}
\\
&\quad
-\dfrac{\Theta_{9,27}^3\overline{\Theta}_{4,9}\overline{\Theta}_{12,27}}
{2q^3\,\Theta_{1,3}\overline{\Theta}_{0,27}\overline{\Theta}_{0,9}\Theta_{4,9}}
+\dfrac{\Theta_{9,27}^3\overline{\Theta}_{2,9}\overline{\Theta}_{6,27}}
{2q^2\,\Theta_{1,3}\overline{\Theta}_{0,27}\overline{\Theta}_{0,9}\Theta_{2,9}}
+\dfrac{q\,\Theta_6^3}{2\Theta_1\Theta_2}.
\end{align*}

Next, we apply Theorem~\thm{thm3p3} to the second Appell--Lerch sum with $z_0=-q^9$, together with \eqn{theta1} and \eqn{theta2}, and simplify to obtain
\begin{align*}
\psi_{{-}_6}(q)
&=
\dfrac{m(1,q^{27},-q^9)}{2q^3}
-m(q^9,q^{27},-1)
+\dfrac{1}{2q^3}
\dfrac{\Theta_{27,81}^3\,\Theta_{9,27}^2}
{\overline{\Theta}_{9,27}^2\,\overline{\Theta}_{1,27}^2}
+\dfrac{\Theta_{9,27}^3\overline{\Theta}_{1,9}\overline{\Theta}_{3,27}}
{2q\,\Theta_{1,3}\overline{\Theta}_{1,27}\overline{\Theta}_{1,9}\Theta_{1,9}}
\nonumber\\
&\quad
-\dfrac{\Theta_{9,27}^3\overline{\Theta}_{4,9}\overline{\Theta}_{12,27}}
{2q^3\,\Theta_{1,3}\overline{\Theta}_{1,27}\overline{\Theta}_{1,9}\Theta_{4,9}}
+\dfrac{\Theta_{9,27}^3\overline{\Theta}_{2,9}\overline{\Theta}_{6,27}}
{2q^2\,\Theta_{1,3}\overline{\Theta}_{1,27}\overline{\Theta}_{1,9}\Theta_{2,9}}
+\dfrac{q\,\Theta_6^3}{2\Theta_1\Theta_2}.
\end{align*}

Using the catalog in Section~\subsect{HMcat}, we obtain
\begin{align*}
\psi_{{-}_6}(q)
=
\dfrac{\psi_6(q^9)}{2q^3}
-\dfrac{\phi_6(q^9)}{2}
+\Psi_{\psi_{{-}_6},3}(q).
\end{align*}

where
\begin{align*}
\Psi_{\psi_{{-}_6},3}(q)
&=
\dfrac{1}{2q^3}
\dfrac{\Theta_{27,81}^3\,\Theta_{9,27}^2}
{\overline{\Theta}_{9,27}^2\,\overline{\Theta}_{1,27}^2}
+\dfrac{1}{2q}
\dfrac{\Theta_{9,27}^3\,\overline{\Theta}_{1,9}\,\overline{\Theta}_{3,27}}
{\Theta_{1,3}\,\overline{\Theta}_{1,27}\,\overline{\Theta}_{1,9}\,\Theta_{1,9}}
-\dfrac{1}{2q^3}
\dfrac{\Theta_{9,27}^3\,\overline{\Theta}_{4,9}\,\overline{\Theta}_{12,27}}
{\Theta_{1,3}\,\overline{\Theta}_{1,27}\,\overline{\Theta}_{1,9}\,\Theta_{4,9}}
\\
&\quad
+\dfrac{1}{2q^2}
\dfrac{\Theta_{9,27}^3\,\overline{\Theta}_{2,9}\,\overline{\Theta}_{6,27}}
{\Theta_{1,3}\,\overline{\Theta}_{1,27}\,\overline{\Theta}_{1,9}\,\Theta_{2,9}}
+\dfrac{q}{2}
\dfrac{\Theta_{6,18}^3}
{\Theta_{1,3}\Theta_{2,6}}.
\end{align*}

We conjecture that
\beq
\Psi_{\psi_{{-}_6},3}(q)
=
\dfrac{\Theta_6^7\Theta_9^7}
{2\Theta_3^8\Theta_{18}^5}
+
q\dfrac{\Theta_6^6\Theta_9^4}
{\Theta_3^7\Theta_{18}^2}
+
2q^2\dfrac{\Theta_6^5\Theta_9\Theta_{18}}
{\Theta_3^6}.
\label{eq:conj:Psipsiminus63}
\eeq

We rewrite this as an identity involving generalized eta-products on
$\Gamma_1(54)$ and compute that $B=-144$.
We verify that the first $145$ terms of the $q$-expansions on both sides agree, and we also check the identity up to $O(q^{252})$.

This proves \eqn{conj:Psipsiminus63}, which establishes \eqn{psiminus6diss3}, giving the required $3$-dissection of $\psi_{{-}_6}(q)$.
\end{proof}
\subsection{Proof of Theorem \thm{ord8diss3}}

\begin{proof}[Proof of \eqn{U08diss3}]
We begin with
$$
U_{0,8}(q)=2m(-q,q^4,-1).
$$

Applying Corollary~\corol{cor3p6} with $n=3$ and $z'=-1$, together with
\eqref{eq:m_inversion} from Proposition~\propo{prop3p1}, and simplifying, we get
\begin{align*}
U_{0,8}(q)
&=
-\dfrac{2m(-q^9,q^{36},-1)}{q}
+2m(-q^{15},q^{36},-1)
+\dfrac{2m(-q^3,q^{36},-1)}{q^3}
\\
&\quad
+\dfrac{2\Theta_{12,36}^3\overline{\Theta}_{3,12}\Theta_{0,36}}
{\Theta_{1,4}\overline{\Theta}_{0,36}\Theta_{3,12}\overline{\Theta}_{0,12}}
-\dfrac{4\Theta_{12,36}^4\overline{\Theta}_{5,12}}
{q^3\Theta_{1,4}\overline{\Theta}_{0,36}\Theta_{3,12}\overline{\Theta}_{4,12}}
\\
&\quad
+\dfrac{2\Theta_{12,36}^4\overline{\Theta}_{1,12}}
{q^2\Theta_{1,4}\overline{\Theta}_{0,36}\Theta_{3,12}\overline{\Theta}_{4,12}}.
\end{align*}

Next, we apply Corollary~\corol{cor3p6} with $n=2$, $z'=-1$, and
$q\mapsto q^3$ to $\phi_6(q)=2m(q,q^3,-1)$, together with
\eqref{eq:m_inversion} from Proposition~\propo{prop3p1}. After simplifying, we obtain
\begin{align*}
\phi_6(q^3)
&=
\dfrac{2m(-q^3,q^{36},-1)}{q^3}
+2m(-q^{15},q^{36},-1)
-\dfrac{2\Theta_{18,54}^3\overline{\Theta}_{3,18}}
{\overline{\Theta}_{3,9}\Theta_{3,18}\overline{\Theta}_{0,18}}
\\
&\quad
-\dfrac{2}{q^3}
\dfrac{\Theta_{18,54}^{3}\,\overline{\Theta}_{6,18}\,\overline{\Theta}_{18,36}}
{\overline{\Theta}_{3,9}\,\overline{\Theta}_{0,36}\,\Theta_{3,18}\,\overline{\Theta}_{9,18}}.
\end{align*}

Combining these expressions gives
\begin{align*}
U_{0,8}(q)
=
\phi_6(q^3)
-\dfrac{U_{0,8}(q^9)}{q}
+\Psi_{U_{0,8},3}(q).
\end{align*}

where
\begin{align*}
\Psi_{U_{0,8},3}(q)
&=
-2\,\dfrac{\Theta_{12,36}^{3}\,\overline{\Theta}_{3,12}\,\Theta_{0,36}}
{\Theta_{1,4}\,\overline{\Theta}_{0,36}\,\Theta_{3,12}\,\overline{\Theta}_{0,12}}
-\dfrac{2}{q^3}
\dfrac{\Theta_{12,36}^{4}\,\overline{\Theta}_{5,12}}
{\Theta_{1,4}\,\overline{\Theta}_{0,36}\,\Theta_{3,12}\,\overline{\Theta}_{4,12}}
+\dfrac{2}{q^2}
\dfrac{\Theta_{12,36}^{4}\,\overline{\Theta}_{1,12}}
{\Theta_{1,4}\,\overline{\Theta}_{0,36}\,\Theta_{3,12}\,\overline{\Theta}_{4,12}}
\\
&\quad
+2\,\dfrac{\Theta_{18,54}^{3}\,\overline{\Theta}_{3,18}}
{\overline{\Theta}_{3,9}\,\Theta_{3,18}\,\overline{\Theta}_{0,18}}
+\dfrac{2}{q^3}
\dfrac{\Theta_{18,54}^{3}\,\overline{\Theta}_{6,18}\,\overline{\Theta}_{18,36}}
{\overline{\Theta}_{3,9}\,\overline{\Theta}_{0,36}\,\Theta_{3,18}\,\overline{\Theta}_{9,18}}.
\end{align*}

We conjecture that
\beq
\Psi_{U_{0,8},3}(q)
=
\dfrac{1}{q}
\dfrac{\Theta_6^2\Theta_{12}\Theta_{36}^5}
{\Theta_3\Theta_{18}^2\Theta_{24}^2\Theta_{72}^2}
+q\dfrac{\Theta_6\Theta_{12}^3\Theta_{18}\Theta_{72}}
{\Theta_3\Theta_{24}^3\Theta_{36}}
+q^3\dfrac{\Theta_3^3\Theta_{12}\Theta_{72}^2}
{\Theta_6^2\Theta_{24}^2\Theta_{36}}.
\label{eq:conj:PsiU083}
\eeq

We rewrite this as an identity involving generalized eta-products on
$\Gamma_1(72)$ and compute that $B=-192$.
We verify that the first $193$ terms of the $q$-expansions on both sides agree,
and we also check the identity up to $O(q^{336})$.

This proves \eqn{conj:PsiU083}, which establishes \eqn{U08diss3}, giving the required $3$-dissection of $U_{0,8}(q)$.
\end{proof}


\begin{proof}[Proof of \eqn{U18diss3}]
We begin with
$$
U_{1,8}(q)=-m(-q,q^4,-q^2).
$$

Applying Corollary~\corol{cor3p6} with $n=3$ and $z'=-1$, together with
\eqref{eq:m_inversion} from Proposition~\propo{prop3p1}, and simplifying, we get
\begin{align*}
U_{1,8}(q)
&=
\dfrac{m(-q^9,q^{36},-1)}{q}
-m(-q^{15},q^{36},-1)
-\dfrac{m(-q^3,q^{36},-1)}{q^3}
\\
&\quad
-\dfrac{1}{q^2}
\dfrac{\Theta_{12,36}^{3}\,\overline{\Theta}_{5,12}\,\Theta_{6,36}}
{\Theta_{1,4}\,\overline{\Theta}_{0,36}\,\Theta_{3,12}\,\overline{\Theta}_{2,12}}
+\dfrac{1}{q^3}
\dfrac{\Theta_{12,36}^{3}\,\overline{\Theta}_{3,12}\,\Theta_{18,36}}
{\Theta_{1,4}\,\overline{\Theta}_{0,36}\,\Theta_{3,12}\,\overline{\Theta}_{6,12}}
\\
&\quad
-\dfrac{1}{q}
\dfrac{\Theta_{12,36}^{3}\,\overline{\Theta}_{1,12}\,\Theta_{6,36}}
{\Theta_{1,4}\,\overline{\Theta}_{0,36}\,\Theta_{3,12}\,\overline{\Theta}_{2,12}}.
\end{align*}

Next, we apply Corollary~\corol{cor3p6} with $n=2$, $z'=-1$, and
$q\mapsto q^3$ to $\phi_6(q)=2m(q,q^3,-1)$, together with
\eqref{eq:m_inversion} from Proposition~\propo{prop3p1}. After simplifying, we obtain
\begin{align*}
\phi_6(q^3)
&=
\dfrac{2m(-q^3,q^{36},-1)}{q^3}
+2m(-q^{15},q^{36},-1)
-\dfrac{2\Theta_{18,54}^3\overline{\Theta}_{3,18}}
{\overline{\Theta}_{3,9}\Theta_{3,18}\overline{\Theta}_{0,18}}
\\
&\quad
-\dfrac{2}{q^3}
\dfrac{\Theta_{18,54}^{3}\,\overline{\Theta}_{6,18}\,\overline{\Theta}_{18,36}}
{\overline{\Theta}_{3,9}\,\overline{\Theta}_{0,36}\,\Theta_{3,18}\,\overline{\Theta}_{9,18}}.
\end{align*}

Combining these expressions gives
\begin{align*}
U_{1,8}(q)
=
-\dfrac{\phi_6(q^3)}{2}
+\dfrac{U_{0,8}(q^9)}{2q}
+\Psi_{U_{1,8},3}(q).
\end{align*}

where
\begin{align*}
\Psi_{U_{1,8},3}(q)
&=
-\dfrac{1}{q^2}
\dfrac{\Theta_{12,36}^{3}\,\overline{\Theta}_{5,12}\,\Theta_{6,36}}
{\Theta_{1,4}\,\overline{\Theta}_{0,36}\,\Theta_{3,12}\,\overline{\Theta}_{2,12}}
+\dfrac{1}{q^3}
\dfrac{\Theta_{12,36}^{3}\,\overline{\Theta}_{3,12}\,\Theta_{18,36}}
{\Theta_{1,4}\,\overline{\Theta}_{0,36}\,\Theta_{3,12}\,\overline{\Theta}_{6,12}}
\\
&\quad
-\dfrac{1}{q}
\dfrac{\Theta_{12,36}^{3}\,\overline{\Theta}_{1,12}\,\Theta_{6,36}}
{\Theta_{1,4}\,\overline{\Theta}_{0,36}\,\Theta_{3,12}\,\overline{\Theta}_{2,12}}
-\dfrac{\Theta_{18,54}^{3}\,\overline{\Theta}_{3,18}}
{\overline{\Theta}_{3,9}\,\Theta_{3,18}\,\overline{\Theta}_{0,18}}
\\
&\quad
-\dfrac{1}{q^3}
\dfrac{\Theta_{18,54}^{3}\,\overline{\Theta}_{6,18}\,\overline{\Theta}_{18,36}}
{\overline{\Theta}_{3,9}\,\overline{\Theta}_{0,36}\,\Theta_{3,18}\,\overline{\Theta}_{9,18}}.
\end{align*}

We conjecture that
\beq
\Psi_{U_{1,8},3}(q)
=
-\dfrac{1}{2q}
\dfrac{\Theta_6^4\,\Theta_9\,\Theta_{24}^2\,\Theta_{36}^4}
{\Theta_{12}^6\,\Theta_{18}^2\,\Theta_{72}^2}
+\dfrac{\Theta_3^3\,\Theta_{24}^2\,\Theta_{36}^5}
{2\,\Theta_{12}^5\,\Theta_{18}^2\,\Theta_{72}^2}
+q\,\dfrac{\Theta_6^4\,\Theta_{24}^3\,\Theta_{36}^2}
{\Theta_3\,\Theta_{12}^6\,\Theta_{72}}.
\label{eq:conj:PsiU183}
\eeq

We rewrite this as an identity involving generalized eta-products on
$\Gamma_1(72)$ and compute that $B=-192$.
We verify that the first $193$ terms of the $q$-expansions on both sides agree,
and we also check the identity up to $O(q^{336})$.

This proves \eqn{conj:PsiU183}, which establishes \eqn{U18diss3}, giving the required $3$-dissection of $U_{1,8}(q)$.
\end{proof}


\begin{proof}[Proof of \eqn{V08diss3}]
We begin with
$$
V_{0,8}(q)=-\dfrac{2\,m(1,q^8,q)}{q} - \dfrac{\Theta_2^3\Theta_4}{\Theta_1^2\Theta_8}.
$$

Applying Corollary~\corol{cor3p6} with $n=3$ and $z'=-1$, together with
\eqref{eq:m_inversion} from Proposition~\propo{prop3p1}, and simplifying, we obtain
\begin{align*}
V_{0,8}(q)
&=
-\dfrac{4m(q^{24},q^{72},-1)}{q} + \dfrac{2m(1,q^{72},-1)}{q^9} + \dfrac{2 \Theta_{24,72}^3\overline{\Theta}_{1,24}\overline{\Theta}_{3,72}}{q^2\, \Theta_{1,8}\overline{\Theta}_{0,72}\overline{\Theta}_{0,24}\Theta_{1,24}} \\
&\quad - \dfrac{2\Theta_{24,72}^3\overline{\Theta}_{9,24}\overline{\Theta}_{27,72}}{q^9 \Theta_{1,8}\overline{\Theta}_{0,72}\overline{\Theta}_{0,24}\Theta_{9,24}}
+ \dfrac{2\,\Theta_{24,72}^3\overline{\Theta}_{7,24}\overline{\Theta}_{21,72}}{q^8\, \Theta_{1,8}\overline{\Theta}_{0,72}\overline{\Theta}_{0,24}\Theta_{7,24}}
- \dfrac{\Theta_2^3\Theta_4}{\Theta_1^2\Theta_8}.
\end{align*}

Next, we apply Theorem~\thm{thm3p3} to the second Appell--Lerch sum with $z_0=-q^{24}$, together with \eqn{theta1} and \eqn{theta2}. After simplification, this gives
\begin{align*}
V_{0,8}(q)
&=
-\dfrac{4m(q^{24},q^{72},-1)}{q}
+ \dfrac{2\,m(1,q^{72},-q^{24})}{q^9}
+ \dfrac{2\Theta_{72,216}^3\Theta_{24,72}^2}{q^9\overline{\Theta}_{24,72}^2\overline{\Theta}_{0,72}^2}
+ \dfrac{2 \Theta_{24,72}^3\overline{\Theta}_{1,24}\overline{\Theta}_{3,72}}{q^2\, \Theta_{1,8}\overline{\Theta}_{0,72}\overline{\Theta}_{0,24}\Theta_{1,24}} \\
&\quad - \dfrac{2\Theta_{24,72}^3\overline{\Theta}_{9,24}\overline{\Theta}_{27,72}}{q^9 \Theta_{1,8}\overline{\Theta}_{0,72}\overline{\Theta}_{0,24}\Theta_{9,24}}
+ \dfrac{2\,\Theta_{24,72}^3\overline{\Theta}_{7,24}\overline{\Theta}_{21,72}}{q^8\, \Theta_{1,8}\overline{\Theta}_{0,72}\overline{\Theta}_{0,24}\Theta_{7,24}}
- \dfrac{\Theta_2^3\Theta_4}{\Theta_1^2\Theta_8}.
\end{align*}

Using the catalog in Section~\subsect{HMcat}, we rewrite this as
\begin{align*}
V_{0,8}(q)
=
\frac{2\psi_6(q^{24})}{q^9}
-
\frac{2\phi_6(q^{24})}{q}
+
\Psi_{V_{0,8},3}(q).
\end{align*}

where
\begin{align*}
\Psi_{V_{0,8},3}(q)
&=
\dfrac{2\Theta_{72,216}^3\Theta_{24,72}^2}{q^9\overline{\Theta}_{24,72}^2\overline{\Theta}_{0,72}^2}
+
\dfrac{2 \Theta_{24,72}^3\overline{\Theta}_{1,24}\overline{\Theta}_{3,72}}{q^2\, \Theta_{1,8}\overline{\Theta}_{0,72}\overline{\Theta}_{0,24}\Theta_{1,24}} \\
&\quad -
\dfrac{2\Theta_{24,72}^3\overline{\Theta}_{9,24}\overline{\Theta}_{27,72}}{q^9 \Theta_{1,8}\overline{\Theta}_{0,72}\overline{\Theta}_{0,24}\Theta_{9,24}}
+
\dfrac{2\,\Theta_{24,72}^3\overline{\Theta}_{7,24}\overline{\Theta}_{21,72}}{q^8\, \Theta_{1,8}\overline{\Theta}_{0,72}\overline{\Theta}_{0,24}\Theta_{7,24}}
-
\dfrac{\Theta_2^3\Theta_4}{\Theta_1^2\Theta_8}.
\end{align*}

We conjecture that
\beq
\Psi_{V_{0,8},3}(q)
=
\dfrac{2}{q}\dfrac{\Theta_6\Theta_{24}^3\Theta_{36}^2}{\Theta_3^2\Theta_{48}^2\Theta_{72}} 
+ \dfrac{\Theta_6^5\Theta_{18}^2\Theta_{24}^4}{\Theta_3^3\Theta_9\Theta_{12}^4\Theta_{48}^2} \nonumber \\
+ 2q\dfrac{\Theta_6^3\Theta_9\Theta_{12}\Theta_{36}}{\Theta_3^3\Theta_{18}\Theta_{24}} 
+ q^3\,\dfrac{\Theta_6^8\Theta_9^2\Theta_{24}\Theta_{36}^2\Theta_{144}}{\Theta_3^4\Theta_{12}^4\Theta_{18}^3\Theta_{48}\Theta_{72}} 
+ 4q^9\,\dfrac{\Theta_6\Theta_{12}^2\Theta_{144}^2}{\Theta_3^2\Theta_{24}\Theta_{72}}.
\label{eq:conj:PsiV083}
\eeq

We rewrite this identity in terms of generalized eta-products on
$\Gamma_1(144)$ and compute that $B=-840$.
We then verify that the first $841$ terms in the $q$-expansions on both sides agree. Furthermore, we check the identity up to $O(q^{1128})$.

This proves \eqn{conj:PsiV083}, which establishes \eqn{V08diss3} and hence gives the required $3$-dissection of $\Psi_{V_{0,8},3}(q)$.
\end{proof}


\begin{proof}[Proof of \eqn{V18diss3}]
We begin with
$$
V_{1,8}(q)=-m(q^2,q^8,q).
$$

Applying Corollary~\corol{cor3p6} with $n=3$ and $z'=-1$, together with
\eqref{eq:m_inversion} from Proposition~\propo{prop3p1}, and simplifying, we obtain
\begin{align*}
V_{1,8}(q)
&=
-\dfrac{m(q^{18},q^{72},-1)}{q^2}
- m(q^{30},q^{72},-1)
+ \dfrac{m(q^6,q^{72},-1)}{q^6}
+ \dfrac{\Theta_{24,72}^3\overline{\Theta}_{7,24}\overline{\Theta}_{3,72}}{q \Theta_{3,8}\overline{\Theta}_{0,72}\overline{\Theta}_{6,24}\Theta_{1,24}}\\
&\quad
- \dfrac{\Theta_{24,72}^3\overline{\Theta}_{9,24}\overline{\Theta}_{27,72}}{q^6\Theta_{3,8}\overline{\Theta}_{0,72}\overline{\Theta}_{6,24}\Theta_{9,24}}
+ \dfrac{\Theta_{24,72}^3\overline{\Theta}_{1,24}\overline{\Theta}_{21,72}}{q^3 \Theta_{3,8}\overline{\Theta}_{0,72}\overline{\Theta}_{6,24}\Theta_{7,24}}.
\end{align*}

Next, we apply Corollary~\corol{cor3p6} with $n=2$, $z'=-1$, and
$q\mapsto -q^6$ to $\phi_6(q)=2m(q,q^3,-1)$, together with
\eqref{eq:m_inversion} from Proposition~\propo{prop3p1}. After simplification, we obtain
\begin{align*}
\phi_6(-q^6)
&=
2m(q^{30},q^{72},-1)
-
\dfrac{2m(q^6,q^{72},-1)}{q^6}
-
\dfrac{2\Theta_{36,108}^3 \Theta_{6,36}}
{\Theta(q^6;-q^{18})\overline{\Theta}_{6,36}\overline{\Theta}_{0,36}} \\
&\quad
+
\dfrac{2q^6 \Theta_{36,108}^3\overline{\Theta}_{48,36}\overline{\Theta}_{36,72}}
{\Theta(q^6;-q^{18})\overline{\Theta}_{0,72}\overline{\Theta}_{6,36}\Theta_{18,36}}.
\end{align*}

Combining these expressions, and using $U_{0,8}(q)=2m(-q,q^4,-1)$, gives
\begin{align*}
V_{1,8}(q)
=
-\dfrac{\phi_6(-q^6)}{2}
-
\dfrac{U_{0,8}(-q^{18})}{2q^2}
+
\Psi_{V_{1,8},3}(q).
\end{align*}

where
\begin{align*}
\Psi_{V_{1,8},3}(q)
&=
\dfrac{\Theta_{24,72}^3\overline{\Theta}_{7,24}\overline{\Theta}_{3,72}}
{q \Theta_{3,8}\overline{\Theta}_{0,72}\overline{\Theta}_{6,24}\Theta_{1,24}}
-
\dfrac{\Theta_{24,72}^3\overline{\Theta}_{9,24}\overline{\Theta}_{27,72}}
{q^6\Theta_{3,8}\overline{\Theta}_{0,72}\overline{\Theta}_{6,24}\Theta_{9,24}} \\
&\quad
+
\dfrac{\Theta_{24,72}^3\overline{\Theta}_{1,24}\overline{\Theta}_{21,72}}
{q^3 \Theta_{3,8}\overline{\Theta}_{0,72}\overline{\Theta}_{6,24}\Theta_{7,24}}
-
\dfrac{2\Theta_{36,108}^3 \Theta_{6,36}}
{\Theta(q^6;-q^{18})\overline{\Theta}_{6,36}\overline{\Theta}_{0,36}} \\
&\quad
+
\dfrac{2q^6 \Theta_{36,108}^3\overline{\Theta}_{48,36}\overline{\Theta}_{36,72}}
{\Theta(q^6;-q^{18})\overline{\Theta}_{0,72}\overline{\Theta}_{6,36}\Theta_{18,36}}.
\end{align*}

We conjecture that
\beq
\Psi_{V_{1,8},3}(q)
=
\dfrac{1}{2q^2}
\dfrac{\Theta_6\,\Theta_{18}^3\,\Theta_{24}^5}
{\Theta_3^2\,\Theta_{12}^2\,\Theta_{36}\,\Theta_{48}^2\,\Theta_{72}}
+
\dfrac{\Theta_9^2\,\Theta_{12}^7}
{2\,\Theta_3^2\,\Theta_6^2\,\Theta_{18}\,\Theta_{24}^3}
+
q^2\,
\dfrac{\Theta_6^4\,\Theta_9\,\Theta_{24}\,\Theta_{36}}
{\Theta_3^3\,\Theta_{12}^2\,\Theta_{18}}.
\label{eq:conj:PsiV183}
\eeq

We rewrite this identity in terms of generalized eta-products on
$\Gamma_1(144)$ and compute that $B=-744$.
We then verify that the first $745$ terms in the $q$-expansions on both sides agree. Furthermore, we check the identity up to $O(q^{1032})$.

This proves \eqn{conj:PsiV183}, which establishes \eqn{V18diss3} and hence gives the required $3$-dissection of $V_{1,8}(q)$.
\end{proof}
\section{Concluding remarks}\label{Concluding remarks}

\begin{enumerate}

\item In view of the corollaries established above, one can derive further congruence relations for various mock theta functions. For example, from \eqref{eq:V0 3n+2}, it follows immediately that
\begin{equation*}
P_{V_{0,8}}(3n+2)\equiv 0 \pmod{2}.
\end{equation*}

\item We observe that the pair $(\rho_6,\lambda_6)$ shares the same even part, while $(\sigma_6,\mu_6)$ shares the same odd part. This naturally leads to the question of whether similar phenomena occur for other mock theta functions. More generally, does there exist an integer $a \in \mathbb{N}$ and a residue class $0 \le r \le a-1$ such that two distinct mock theta functions have identical coefficients along the arithmetic progression $an + r$?

\item From equations \eqref{A2}, \eqref{eq:V_18 2n}, and \eqn{eta1diss2} we obtain
\begin{equation*}
    \sum_{n \ge 0} P_{V_{1,8}}(2n)\, q^{n} = A_2(q) = \sum_{n \ge 0} P_{A_2}(n)\, q^{n}.
\end{equation*}
Thus, in response to the above question, we see that $A_2(q)$ and $V_{1,8}(q)$ agree under the $2n$-projection of $V_{1,8}(q)$.
\item Using the method developed in this paper, one can also obtain 
$m$-dissections of other mock theta functions. We plan to pursue this 
direction in subsequent work.   

\item We conclude this paper with a short and elementary proof of another
$2$-dissection for $V_{0,8}(q)$. The argument is simple and direct.

\begin{proof}
From \cite[Equation (2.14)]{mortenson2024ramanujan}, we have
\[
qB_2(q)-2A_2(-q^4)
=
q\,\frac{\Theta_2\,\Theta_4^5\,\Theta_{16}^2}
{\Theta_1^2\,\Theta_8^5}.
\]

Replacing $q$ by $q^2$ yields
\begin{equation}
q^2B_2(q^2)-2A_2(-q^8)
=
q^2\,
\frac{\Theta_4\,\Theta_8^5\,\Theta_{32}^2}
{\Theta_2^2\,\Theta_{16}^5}.
\label{eq:v01}
\end{equation}

We also have from \cite[Equations (3) and (6)]{mcintosh2007second} that
\begin{align}
V_0(q)-V_0(-q)
&=
4q\,B(q^2),
\label{eq:v02}
\\
V_0(q)+V_0(-q)
&=
2\,\frac{\Theta_4^8}
{\Theta_2^4\,\Theta_8^3}.
\label{eq:v03}
\end{align}

Adding \eqref{eq:v02} and \eqref{eq:v03}, we obtain
\[
V_0(q)
=
2q\,B(q^2)
+
\frac{\Theta_4^8}
{\Theta_2^4\,\Theta_8^3}.
\]

Hence,
\begin{align*}
V_0(q)-\frac{4}{q}A_2(-q^8)
&=
2q\,B(q^2)
+
\frac{\Theta_4^8}
{\Theta_2^4\,\Theta_8^3}
-
\frac{4}{q}A_2(-q^8)
\\
&=
\frac{2}{q}
\left(
q^2B_2(q^2)-2A_2(-q^8)
\right)
+
\frac{\Theta_4^8}
{\Theta_2^4\,\Theta_8^3}
\\
&=
2q\,
\frac{\Theta_4\,\Theta_8^5\,\Theta_{32}^2}
{\Theta_2^2\,\Theta_{16}^5}
+
\frac{\Theta_4^8}
{\Theta_2^4\,\Theta_8^3},
\end{align*}
where the final equality follows from \eqref{eq:v01}. Therefore,
\[
V_0(q)
=
\frac{4}{q}A_2(-q^8)
+
2q\,
\frac{\Theta_4\,\Theta_8^5\,\Theta_{32}^2}
{\Theta_2^2\,\Theta_{16}^5}
+
\frac{\Theta_4^8}
{\Theta_2^4\,\Theta_8^3}.
\]
This completes the proof.
\end{proof}

\item In this paper, we obtain complete 2-dissections for the second-order mock theta functions $A_2(q)$, $B_2(q)$, and $\mu_2(q)$; the sixth-order mock theta functions $\rho_6(q)$, $\lambda_6(q)$, $\sigma_6(q)$, and $\mu_6(q)$; and the eighth-order mock theta functions $V_{0,8}(q)$ and $V_{1,8}(q)$.

For 3-dissections, we are able to obtain results only for the second-order mock theta functions $A_2(q)$, $B_2(q)$, and $\mu_2(q)$; the sixth-order mock theta functions $\psi_6(q)$, $\rho_6(q)$, $\lambda_6(q)$, and $\psi_{-6}(q)$; and the eighth-order mock theta functions $U_{0,8}(q)$, $U_{1,8}(q)$, $V_{0,8}(q)$, and $V_{1,8}(q)$.

We were unable to obtain 2-dissections for the following sixth-order mock theta functions:
\begin{equation*}
\phi_6(q),\ \psi_6(q),\ \gamma_6(q),\ \phi_{-6}(q),\ \psi_{-6}(q),
\end{equation*}
and for the following eighth-order mock theta functions:
\begin{equation*}
S_{0,8}(q),\ S_{1,8}(q),\ T_{0,8}(q),\ T_{1,8}(q),\ U_{0,8}(q),\ U_{1,8}(q).
\end{equation*}

Similarly, for 3-dissections, we were unable to obtain results for the sixth-order mock theta functions
\begin{equation*}
\phi_6(q),\ \sigma_6(q),\ \mu_6(q),\ \gamma_6(q),\ \phi_{-6}(q),
\end{equation*}
and for the eighth-order mock theta functions
\begin{equation*}
S_{0,8}(q),\ S_{1,8}(q),\ T_{0,8}(q),\ T_{1,8}(q).
\end{equation*}

Therefore, we leave these unresolved dissections as open problems for interested readers, hoping that future investigations of these cases will produce genuinely new dissections that have not previously appeared in the literature.

\end{enumerate}

\newpage

\appendix

\section{A General Algorithm for Proving Theta Identities}\label{sec:algorithm}

In this appendix, we review a general modular-function method for proving the theta identities given in \cite{frye2019automatic}. The argument is based on the Valence Formula together with standard results on generalized eta-products.

Suppose we wish to prove an identity involving a finite product or quotient of theta functions.

\medskip\noindent
\textbf{$\theta$-Step 1.}
Rewrite the identity as
\[
\sum_{j=1}^r \alpha_j F_j(\tau)+1=0,
\]
where each $F_j$ is a generalized eta-product. Let
\[
g(\tau)=\sum_{j=1}^r \alpha_j F_j(\tau)+1.
\]
It suffices to show that $g(\tau)\equiv 0$.

\medskip\noindent
\textbf{$\theta$-Step 2.}
Show that each $F_j(\tau)$ is a modular function on $\Gamma_1(N)$ for some $N$.

\medskip\noindent
\textbf{$\theta$-Step 3.}
Determine a complete set $\mathcal{S}_N$ of inequivalent cusps of $\Gamma_1(N)$ together with their widths $\kappa(\zeta,\Gamma)$.

\medskip\noindent
\textbf{$\theta$-Step 4.}
For each $F_j$ and $\zeta\in\mathcal{S}_N$, compute
\[
\operatorname{ORD}(F_j,\zeta,\Gamma)
=\kappa(\zeta,\Gamma)\operatorname{ord}(F_j,\zeta).
\]

\medskip\noindent
\textbf{$\theta$-Step 5.}
Define
\[
B=\sum_{\substack{\zeta\in\mathcal{S}_N\\ \zeta\neq\infty}}
\min\Big(\{\operatorname{ORD}(F_j,\zeta,\Gamma):1\le j\le r\}\cup\{0\}\Big).
\]

\medskip\noindent
\textbf{$\theta$-Step 6.}
Compute the $q$-expansion of $g(\tau)$ at $\infty$ and verify that
\[
\operatorname{ORD}(g,\infty,\Gamma_1(N))>-B.
\]
Then $g(\tau)\equiv 0$ by the Valence Formula.

\end{document}